\documentclass{mcom-l}

\usepackage{cite}
\usepackage{amsthm,amsmath,amssymb,amsfonts}
\usepackage{graphicx}
\usepackage{textcomp}
\def\BibTeX{{\rm B\kern-.05em{\sc i\kern-.025em b}\kern-.08em
		T\kern-.1667em\lower.7ex\hbox{E}\kern-.125emX}}
\usepackage{bm}
\usepackage{cancel}
\usepackage{xcolor}
\usepackage{cleveref}
\usepackage{commath}
\usepackage{empheq}
\usepackage{commath}
\usepackage{mathrsfs}
\usepackage{bm}
\usepackage{dsfont}
\usepackage{algorithm}
\usepackage[noend]{algpseudocode}
\usepackage{enumerate}
\usepackage[most]{tcolorbox}
\usepackage{multirow}
\usepackage{subcaption}
\usepackage[justification=centering]{caption}
\newtcolorbox{whiteFrameResultBeforeAfter}{
	colframe=black!20!white,
	top=2mm, bottom=3 mm, left=0mm, right=0mm,
	arc=0mm,
}
\DeclareMathOperator{\ran}{Ran}
\DeclareMathOperator*{\argmin}{argmin}

\newtheorem{theorem}{Theorem}[section]
\newtheorem{lemma}[theorem]{Lemma}
\newtheorem{corollary}[theorem]{Corollary}

\theoremstyle{definition}

\newtheorem{assumption}[theorem]{Assumption}

\theoremstyle{remark}
\newtheorem{remark}[theorem]{Remark}

\numberwithin{equation}{section}

\begin{document}

\title{Minimal $\ell^2$ Norm Discrete Multiplier Method}


\author[Erick Schulz]{Erick Schulz$^1$}
\address{$^1$Seminar in Applied Mathematics, Swiss Federal Institute of Technology Zurich}
\curraddr{Rämistrasse 101, CH-8092 Zurich, Switzerland}
\email{erick.schulz@sam.math.ethz.ch}
\thanks{$^1$Seminar in Applied Mathematics, ETH Z\"{u}rich, Switzerland (erick.schulz@sam.math.ethz.ch)}

\author[Andy T. S. Wan]{Andy T. S. Wan$^2$}
\address{$^2$Department of Mathematics and Statistics, University of Northern British Columbia}
\curraddr{3333 University Way, Prince George, BC V2N 4Z9, Canada}
\email{andy.wan@unbc.ca}
\thanks{$^2$Department of Mathematics \& Statistics, University of Northern British Columbia, Canada (andy.wan@unbc.ca)}



\begin{abstract}
We introduce an extension to the Discrete Multiplier Method (DMM) \cite{wan2017conservative}, called Minimal $\ell_2$ Norm Discrete Multiplier Method (MN-DMM), where conservative finite difference schemes for dynamical systems with multiple conserved quantities are constructed procedurally, instead of analytically as in the original DMM. For large dynamical systems with multiple conserved quantities, MN-DMM alleviates difficulties that can arise with the original DMM at constructing conservative schemes which satisfies the discrete multiplier conditions. In particular, MN-DMM utilizes the right Moore-Penrose pseudoinverse of the discrete multiplier matrix to solve an underdetermined least-square problem associated with the discrete multiplier conditions. We prove consistency and conservative properties of the MN-DMM schemes. We also introduce two variants -- Mixed MN-DMM and MN-DMM using Singular Value Decomposition -- and discuss their usage in practice. Moreover, numerical examples on various problems arising from the mathematical sciences are shown to demonstrate the wide applicability of MN-DMM and its relative ease of implementation compared to the original DMM.
\end{abstract}


\maketitle

\section{Introduction} In recent decades, numerical methods which preserve intrinsic geometric structures of dynamical systems have gained considerable interest. Geometric numerical integrators are numerical methods which preserve underlying geometric features of solutions between successive time steps. An extensive summary of relevant literature is presented in \cite{hair06Ay}. In addition to striving for the traditional goals of high order accuracy, stability and ease of implementation, geometric numerical integrators seek to respect inherent geometric structures of dynamical systems to provide more accurate and stable solutions over \emph{long-term} integration. Examples of geometric numerical integrators include symplectic integrators which preserve the symplectic two-forms associated with Hamiltonian flows \cite{hair06Ay}, variational integrators which mimic the action principles of Lagrangian systems at the discrete level \cite{MarWes01} and Lie group integrators which compose discrete Lie group actions to approximate underlying continuous Lie group flows \cite{IMNZ00}. 

Another important class of geometric numerical integrators is conservative integrators, which preserve conserved quantities or invariants associated with the underlying dynamics. In general, similar to symplectic methods \cite{CalHai95}, conservative numerical methods can have favorable long-term stability properties \cite{wan2018stability} over traditional numerical methods. Typically, such quantities include energy and momentum, but nontrivial time-dependent conserved quantities can also exist in dissipative systems \cite{wan2017conservative}. Unfortunately, traditional numerical methods do not in general preserve all forms of conserved quantities. For instance, the barrier theorem by \cite{KangZaijiu95} states that no consistent Runge-Kutta method can preserve all polynomial invariants. Thus, non-traditional numerical methods are needed to preserve general forms of conserved quantities. Within the literature of geometric numerical integration, there are a few general classes of conservative integrators, such as projection methods \cite{hair06Ay}, discrete gradient methods \cite{McL99} and more recently Discrete Multiplier Methods (DMM) \cite{wan2017conservative}, which we briefly review next. 

With projection methods, it is customary to first employ a traditional explicit integrator to advance one step in time, then to subsequently project the resulting numerical approximation onto the level set of the conserved quantities by solving a constrained optimization problem \cite{hair06Ay}. As discussed in \cite{wan2018stability}, while projection methods are general conservative integrators, the projection step can become problematic if the level set of the conserved quantities contain connected components which are nearby each other. Indeed, if the time step size is not sufficiently small to account for small distances between neighboring connected components, the projection step may bring the numerical solution to the wrong connected component, leading to incorrect long-term trajectories.

The discrete gradient method exploits the fact that dynamical systems with conserved quantities can be expressed in a skew-symmetric gradient form \cite{McL99}. This can then be used to derive conservative schemes using discrete gradient approximations. While the discrete gradient method is best suited for dynamical systems which naturally comes in such a skew-symmetric gradient representation, such as Hamiltonian systems, transforming a general dynamical system and utilizing the resulting skew-symmetric gradient form is not always straightforward in practice. Specifically, one drawback of the discrete gradient method is that the rank of the skew-symmetric tensor increases with the number of conserved quantities, making its applicability impractical for large dynamical systems with multiple invariants.

 The Discrete Multiplier Method (DMM) was introduced in \cite{wan2017conservative} as a new class of general conservative integrators that can preserve multiple conserved quantities of arbitrary forms up to machine precision. The main idea behind DMM is to discretize the so-called conservation law multiplier associated with the conserved quantities in such a way that discrete chain rules and other compatibility conditions are satisfied. In contrast to the discrete gradient method, DMM can work directly with the desired dynamical system,  without having to reformulate the differential equations. Moreover, DMM requires only working with the so-called discrete multiplier matrix, whose number of rows increases with the number of conserved quantities while retaining a \emph{constant tensor rank of two}. Such conservative integrators have recently been applied to a wide range of problems from the mathematical sciences, including many-body systems \cite{manybody2022}, vortex-blob models \cite{vortex2022}, and piecewise smooth systems \cite{HWW21}. In addition, for some applications such as Hamiltonian Monte Carlo \cite{chmc22}, the gradient-free nature of DMM is advantageous over other conservative methods which require computation of the gradients of the conserved quantities. 
 
 Despite the wide applicability of DMM, there remains practical challenges when applying DMM on large dynamical systems with multiple conserved quantities. Specifically for each dynamical system, DMM proceeds in two main stages: First, derive an analytic conservative scheme using DMM; Second, solve the associated implicit conservative scheme. In this work, we extend the DMM framework by implicitly defining conservative schemes via a Moore-Penrose pseudoinverse of the associated discrete multiplier matrix. In doing so, DMM conservative schemes are constructed \emph{procedurally} and solved \emph{simultaneously}, without the need to first derive analytic conservative schemes. This extension of DMM widens its applicability to more complex dynamical systems and semi-discretizations of partial differential equations with conserved quantities. 
 
 This paper is organized as follows. In \Cref{sec:background}, we give a brief overview of the background material of DMM and introduce the relevant notations used throughout the paper. We then introduce the \emph{Minimal $\ell_2$ Norm Discrete Multiplier Method} (MN-DMM) in \Cref{sec: Minimal Norm DMM}. Consistency and conservative properties of the implicitly defined schemes are established. In \Cref{sec: Practical Implementation}, we discuss practical issues that can arise in solving the MN-DMM schemes using the \emph{Direct MN-DMM algorithm} via fixed point iterations. We prove convergence under appropriate conditions. Furthermore, we introduce in \Cref{sec: Practical Implementation} two variants of the Direct MN-DMM algorithm, called \emph{Mixed MN-DMM algorithm} and \emph{Mixed MN-DMM algorithm using Singular Value Decomposition}. These two variants alleviate potential drawbacks with the Direct MN-DMM algorithm. In \Cref{sec: Numerical results}, numerical comparisons between the various MN-DMM approaches and traditional methods are presented for five examples chosen from a wide range of applications in the mathematical sciences. These includes Lotka--Volterra systems, the planar restricted three-body problem, the Lorenz system, the spherical point vortex problem and the evolution of geodesic curves in Schwarzschild geometry. 
\vskip -4mm
\section{Background material}\label{sec:background}

We retain most of the notations of the Discrete Multiplier Method from \cite{wan2017conservative}. For details on the theoretical developments of DMM, see \cite{wan2017conservative} and \cite{wan2018stability}. In this section, we summarize the content of these articles by stating some basic definitions along with a few necessary results.


\subsection{Notation}
Throughout this paper, the integers $m,n,p,r\in\mathbb{N}$ are strictly positive. We denote open subsets of $\mathbb{R}^n$ by $U$, $U^{(1)}$, $U^{(2)}$, etc. We write $U^{r}$ for the Cartesian product of $r$ copies of $U$.

By $f\in C^p(U\rightarrow \mathbb{R}^m)$, we mean that the function $f$ from $U$ to $\mathbb{R}^m$ is at least $p$ times continuously differentiable. We use a \textbf{bold} font to distinguish vector quantities from scalars. The Jacobian matrix of a differentiable vector-valued function $\bm{f}$ is denoted by $\smash{\partial_{\bm{x}}\bm{f}:=\left[\partial f_i/\partial x_j\right]}$. 

 Let $I\subset\mathbb{R}$ be an open time interval. We adopt Newton's notation $\dot{\bm{x}}$ for the time derivative of a curve $\bm{x}\in C^1(I\rightarrow U)$. If $\bm{x}\in C^p(I\rightarrow U)$, then $\bm{x}^{(p)}$ stands for its $p$-th time derivative. We use $D_t\bm{\psi}$ to distinguish the total time derivative  of a vector-valued function $\bm{\psi}\in C^1(I\times U\rightarrow\mathbb{R}^m)$ from its partial time derivative $\partial_t\bm{\psi}$.

The vector space of $m\times n$ real matrices is written as $M_{m\times n}(\mathbb{R})$. It is equipped with the operator norm, which we denote by $\norm{\cdot}_{m\times n}$. A superscript `$\top$' indicates the transpose of a matrix quantity, e.g. $\Lambda^{\top}$. We indicate the dependence of strictly positive constants in parentheses, e.g. $C\left(\Lambda\right)$. These constants are generic and should generally not be considered equal between different results.

\subsection{Review on conservation law multipliers}\label{sec: Review on conservation law multipliers}
Next, we briefly review the theory of conservation law multipliers for first-order quasi-linear systems of ordinary differential equations---recall all quasi-linear systems can be made first-order by adding more variables. More precisely, for $p=1,2,\dots$ and a source function $\bm{f}\in C^{p-1}(I\times U\rightarrow\mathbb{R}^n)$ with Lipschitz continuity in $U$, consider the continuous dynamical system $\bm{F}:I\times U\times U^{(1)}\rightarrow \mathbb{R}^n$ given by the initial value problem
\begin{align}\label{eq: quasilinear ODE}
\arraycolsep=1.4pt
\begin{array}{rl}
		    \bm{F}(t,\bm{x}(t),\dot{\bm{x}}(t)):= \dot{\bm{x}}(t)-\bm{f}(t,\bm{x}(t))  &=\bm{0},\\
		\bm{x}(t^0)&=\bm{x}^0.
\end{array}
\end{align}
It is a classical result that there exists a unique solution $\bm{x}(t)=(x_1(t),...,x_n(t))$ of class $C^p$ in a neighborhood of any initial condition $(t^0,\bm{x}^0)\in I\times U$. For simplicity, we will always assume from now on that $I$ is a maximal interval of existence.

A function $\bm{\psi}\in C^1(I\times U\rightarrow\mathbb{R}^m)$ is called a conserved quantity of $\bm{F}$ if
\begin{equation}\label{eq: def conserved quantity}
D_t\bm{\psi}(t,\bm{x}(t))=\bm{0}
\end{equation}
for all $\bm{x}\in C^{p}(I\rightarrow U)$ such that $\bm{F}(t,\bm{x}(t),\dot{\bm{x}}(t))= \bm 0$. In other words, a conserved quantity
remains constant along \emph{solutions} of \eqref{eq: quasilinear ODE}. In principle, a conserved quantity can depend on higher-order time derivatives of $\bm{x}$ also, but these can always be reformulated as $\bm{\psi}\left(t,\bm{x}\right)$ by substituting the relation $\dot{\bm{x}}=\bm{f}(t,\bm{x}(t))$ and its differential consequences, as shown in \cite[Sec. 3.1]{wan2017conservative}. Thus, without loss of generality, we can focus on conserved quantities of such form.

We say that a matrix-valued function 
$
\Lambda\in C(I\times U \times U^{(1)}\rightarrow M_{m\times n}(\mathbb{R}))
$
is a conservation law multiplier of $\bm{F}$ if there exists $\bm{\psi}\in C^1(I\times U \rightarrow \mathbb{R}^m)$ satisfying
\begin{equation}\label{eq: def conservation law multiplier}
\Lambda (t,\bm{x}(t),\dot{\bm{x}}(t)) \bm{F}(t,\bm{x}(t),\dot{\bm{x}}(t))=D_t\bm{\psi}(t,\bm{x}(t))
\end{equation}
for all $\bm{x}\in C^1(I\rightarrow U)$. We insist that \eqref{eq: def conservation law multiplier} must hold for all \emph{arbitrary} differentiable functions---not only for solutions of $\bm{F}$ as previously required for \eqref{eq: def conserved quantity}. 

In general, there can be many different conservation law multipliers satisfying \eqref{eq: def conservation law multiplier} for the same $\bm{\psi}$. However, the following theorem guarantees that there is a one-to-one correspondence between conservation law multipliers of the form $\Lambda(t,\bm{x})$ and zero-order conserved quantities of $\bm{F}$, cf. \cite[Thm. 4]{wan2017conservative}. 

\begin{theorem}\label{thm: correspondence}
	Let $\bm{\psi}\in C^1(I\times U \rightarrow \mathbb{R}^m)$. There exists a unique conservation law multiplier $\Lambda\in C(I\times U \rightarrow M_{m\times n}(\mathbb{R}))$ of $\bm{F}$ associated with the function $\bm{\psi}$ if and only if $\bm{\psi}$ is a conserved quantity of $\bm{F}$. If so, the correspondence identities
	\begin{subequations}
		\begin{align}
			\Lambda(t,\bm{x}) &= \partial_{\bm{x}}\bm{\psi}(t,\bm{x}),\label{eq: continuous correspondence identity a}\\
			\Lambda(t,\bm{x})\bm{f}(t,\bm{x}) & =-\partial_t\bm{\psi}(t,\bm{x}),\label{eq: continuous correspondence identity b}
		\end{align}
	\end{subequations}
	are satisfied for any arbitrary function $\bm{x}\in C^1(I\rightarrow U)$.
\end{theorem}

We will commonly refer to \eqref{eq: continuous correspondence identity a} and \eqref{eq: continuous correspondence identity b} as \emph{multiplier conditions}. Importantly, \eqref{eq: continuous correspondence identity a} explicitly characterizes the conservation law multiplier.

For the purpose of deriving conservative schemes, there is some freedom in choosing the dimension of $\bm{\psi}$. For a given dynamical system, the components of the vector-valued function $\bm{\psi}$ consist of known conserved quantities of interest. How many are to be preserved using DMM is up to the one's discretion. In practice, there are typically much fewer conserved quantities than the dimension of $\bm{F}$. We thus take for granted the following assumption.

\begin{assumption}\label{assum: nb of conserved quantities}
We suppose that $m<n$ and assume that $\Lambda(t,\bm{x})$ has full row rank within $I\times U$.
\end{assumption}

\begin{remark}
Notice that $\Lambda(t,\bm{x})$ having full row rank in \Cref{assum: nb of conserved quantities} is equivalent to the conserved quantities being linearly independent on $I\times U$.
\end{remark}

\subsection{Review of DMM}\label{sec: Review of DMM}
The idea behind DMM is to provide a discrete framework which preserves the structure of the continuous theory of conservation law multipliers from \Cref{sec: Review on conservation law multipliers}. Specifically, it establishes discrete analogues of the multiplier conditions \eqref{eq: continuous correspondence identity a} and \eqref{eq: continuous correspondence identity b}.

Let $t^0< t^1 < ... < t^k < ...$ be a sequence in $I$ having a largest time step of size $\tau=\sup_{k}(t^{k+1}-t^k)<\infty$. 
Let $W$ be a finite dimensional normed vector space. A $r$-step function 
$
\bm{g}^{\tau}\in C^{p+q}(I\times U^{r+1}\rightarrow W)
$
is said to be consistent of order $q$ to a function $\bm{g}\in C^{p+q}(I\times U\times U^{(1)}\times...\times U^{(q)}\rightarrow W)$ if for any $\bm{x}\in C^{p+q}\left(I\rightarrow U\right)$, there exists a constant $C(\bm{g},\bm{x})>0$ independent of $\tau$ such that
\begin{equation}\label{eq: consistency estimate def}
\norm{\bm{g}(t^k,\bm{x}(t^k),...,\bm{x}^{(p)}(t^k))-\bm{g}^{\tau}(t^k,\bm{x}(t^{k+1}),...,\bm{x}(t^{k-r+1}))}_W\leq C(\bm{g},\bm{x})\,\tau^q.
\end{equation}
If so, we simply write $\bm{g}^{\tau}=\bm{g}+\mathcal{O}(\tau^q)$. This definition is general enough to provide a notion of consistency for both vector-valued and matrix-valued quantities. In the following sections, we will encounter $W=\mathbb{R}^m$ and $M_{m\times n}(\mathbb{R})$.

Denote the approximation at time $t_k$ of the exact solution $\bm x(t^k)$ by $\bm{x}^k$. Let $\bm{F}^{\tau}$ be a consistent $r$-step function to $\bm{F}$ and suppose that $\bm{\psi}^{\tau}$ is a consistent $\left(r-1\right)$-step function to $\bm{\psi}$. We say that the $r$-step method $\bm{F}^{\tau}$ is conservative in $\bm{\psi}^{\tau}$ if 
\begin{equation}\label{eq: discrete conservation}
\bm{\psi}^\tau(t^k,\bm{x}^{k},...,\bm{x}^{k-r+1})=\bm{\psi}^\tau(t^{k+1},\bm{x}^{k+1},...,\bm{x}^{k-r+2})
\end{equation}
whenever $\bm{x}^{k+1}$ satisfies $\bm{F}^{\tau}\left(t^k,\bm{x}^{k+1},...,\bm{x}^{k-r+1}\right)=\bm{0}$.

When $D_t^{\tau}\bm{\psi}$ is an $r$-step function consistent to $D_t\bm{\psi}$, we say that it is constant compatible with $\bm{\psi}^{\tau}$ if $D^{\tau}_t\bm{\psi}\left(t^k,...,\bm{x}^{k+1},...,\bm{x}^{k-r+1}\right)=\bm{0}$ implies that \eqref{eq: discrete conservation} holds.

\begin{assumption}\label{assum: consistency}
Henceforth, we will always suppose that $\bm{f}^{\tau}$, $D^{\tau}_t\bm{x}$, $D^{\tau}_t\bm{\psi}$, $\partial^{\tau}_t\bm{\psi}$ and $\Lambda^{\tau}$ are $r$-step functions consistent of order $q$ respectively to $\bm{f}$, $\dot{\bm{x}}$, $D_t\bm{\psi}$, $\partial_t\bm{\psi}$ and $\Lambda$, where $\Lambda$ is a conservation law multiplier of $\bm{F}$ associated with the conserved quantity $\bm{\psi}$. We assume that $D^{\tau}_{t}\bm{\psi}$ is constant compatible with a discrete ($r-1$)-step function $\bm{\psi}^\tau$.
\end{assumption}


The following theorem is the heart of DMM, cf. \cite[Thm. 4.5]{wan2017conservative}.

\begin{theorem}\label{thm: conservative scheme theory}
	 Let $\bm{f}^{\tau}_{\text{\emph{\scalebox{.7}{DMM}}}}$ be a $r$-step function consistent of order $q$ to $\bm{f}$. Under assumptions \ref{assum: nb of conserved quantities} and \ref{assum: consistency}, if the discrete compatibility conditions
	\begin{subequations}
		\begin{align}
		\Lambda^{\tau}D^{\tau}_t\bm{x}&= D^{\tau}_t\bm{\psi}-\partial^{\tau}_t\bm{\psi}, \label{eq: discrete correspondence a}\\
		\Lambda^{\tau}\bm{f}^{\tau}_{\text{\emph{\scalebox{.7}{DMM}}}}&= - \partial_t^{\tau}\bm{\psi},\label{eq: discrete correspondence b}
		\end{align}
	\end{subequations}
hold for all $(t^k,\bm{x}^{k+1},...,\bm{x}^{k-r+1})\in I\times U^{r+1}$ satisfying
\begin{equation*}\label{F r-step method}
    \bm{F}_{\text{\emph{\scalebox{.7}{DMM}}}}^{\tau}(t^k,\bm{x}^{k+1},...,\bm{x}^{k-r+1}) = \bm{0},
\end{equation*}
where $$\bm{F}_{\text{\emph{\scalebox{.7}{DMM}}}}^{\tau}:=D^{\tau}_t\bm{x}-\bm{f}^{\tau}_{\text{\emph{\scalebox{.7}{DMM}}}},$$ then the $r$-step method defined by \eqref{F r-step method} is conservative in $\bm{\psi}^{\tau}$. Moreover, it is consistent of at least order $q$ to $\bm{F}$, and for any sufficiently differentiable arbitrary function $\bm{x}$, the discrete quantities satisfy
	\begin{subequations}
		\begin{align*}
		\Lambda^{\tau}D^{\tau}_t\bm{x} - D^{\tau}_t\bm{\psi} - \partial^{\tau}_t\bm{\psi} &= \mathcal{O}(\tau^q), 
		\\
		\Lambda^{\tau}\bm{f}^{\tau}_{\text{\emph{\scalebox{.7}{DMM}}}}+ \partial^{\tau}_t \bm{\psi} &= \mathcal{O}(\tau^q). 
		\end{align*}
	\end{subequations}
\end{theorem}

\begin{remark}
The discrete compatibility condition \eqref{eq: discrete correspondence a} corresponds \emph{implicitly} to \eqref{eq: continuous correspondence identity a} by the chain rule:
\begin{equation*}
    \Lambda(t,\bm{x})\dot{\bm{x}}= (\partial_{\bm{x}}\bm{\psi})D_t\bm{x}=(\partial_{\bm{x}}\bm{\psi})D_t\bm{x}+ \partial_t\bm{\psi} -\partial_t\bm{\psi}=D_t\bm{\psi}-\partial_t\bm{\psi}.
\end{equation*}
\end{remark}

\section{Minimal $\ell_2$ Norm DMM}\label{sec: Minimal Norm DMM}
 So far, the construction of conservative schemes using DMM reduces to satisfying the discrete multiplier conditions \eqref{eq: discrete correspondence a} and \eqref{eq: discrete correspondence b}. While \eqref{eq: discrete correspondence a} can be resolved through the use of discrete chain rules as described in \cite[Thm. 22]{wan2017conservative}, resolving \eqref{eq: discrete correspondence b} relies on the local solvability of $\Lambda^\tau$. As discussed in \cite[Thm. 20]{wan2017conservative}, \eqref{eq: discrete correspondence b} can be satisfied by locally inverting an $m\times m$ submatrix $\tilde{\Lambda}^\tau$ of $\Lambda^\tau$ for a given dynamical system with $m$ conserved quantities. However, this traditional approach of DMM has two main drawbacks.
\begin{enumerate}
    \item First, analytical matrix inversion of a submatrix of $\Lambda^\tau$ becomes difficult, if not impractical, as $m$ increases. As it will be highlighted in examples later in \Cref{sec: Numerical results}, even a small number of conserved quantities can pose a significant challenge to construct conservative schemes using the traditional DMM approach.
    \item Second, due to the local nature of the rank of the submatrix $\tilde{\Lambda}^\tau$, its invertibility may vary depending on the phase space region where the conservative scheme is to be evaluated, making the traditional DMM approach cumbersome to implement for dynamical systems with complex phase spaces.
\end{enumerate}  
Indeed, there are alternate techniques, such as the \emph{method of undetermined coefficients} used in \cite{wan2017conservative} and \cite{ manybody2022}, that could alleviate some of these difficulties, but it still relies on the need to construct analytic conservative scheme, which can be difficult to apply for large dynamical systems with multiple conserved quantities. 

In this section, we tackle the problem of solving \eqref{eq: discrete correspondence b} systematically for large dynamical systems with multiple conserved quantities, without the need to construct analytic conservative schemes. The proposed new approach paves the way for the procedural construction of globally defined conservative schemes using DMM. Thus, this leads to a promising starting point for conservative discretizations of large dynamical systems which can only be evaluated procedurally and also in semi-discretization of partial differential equations.

\subsection{Minimal $\ell^2$ Norm Discrete Multiplier Method}\label{sec: Minimal l2 Norm Discrete Multiplier Method}
We define the \emph{Minimal $\ell^2$ Norm Discrete Multiplier Method}, or \emph{Minimal Norm DMM} (MN-DMM), as the conservative scheme
\begin{equation}\label{eq: modified f by right inverse}
\bm{f}^{\tau}_{\text{\scalebox{.7}{MN}}} := \bm{f}^{\tau} -(\Lambda^{\tau})^+(\Lambda^{\tau}\bm{f}^\tau+\partial^{\tau}_t\bm{\psi}),
\end{equation}
where $\bm{f}^{\tau}$ is \emph{any} consistent scheme to $\bm{f}$ and 
$
(\Lambda^{\tau})^+=(\Lambda^{\tau})^\top(\Lambda^{\tau}(\Lambda^{\tau})^\top)^{-1}$ is the unique \emph{right} Moore-Penrose pseudoinverse of $\Lambda^{\tau}$. By construction, it can be readily check that $\bm{f}^{\tau}_{\text{\scalebox{.7}{MN}}}$ satisfies \eqref{eq: discrete correspondence b}. We will discuss the theoretical analysis of the implicit scheme $\bm{f}^{\tau}_{\text{\scalebox{.7}{MN}}}$ shortly in \Cref{sec:theory}, where we will show that it is indeed conservative and well-defined for sufficiently small $\tau$. Different algorithmic choices for the practical evaluation of the second term on the right-hand side of \eqref{eq: modified f by right inverse} will be discussed in \Cref{sec: Practical Implementation}. 

Let us motivate the expression of \eqref{eq: modified f by right inverse} in two ways and the reasons for its name\footnote{For a general introduction to both \emph{under}determined and \emph{over}determined $\ell^2$ minimization problems, see the first chapters of the monograph \cite[Chap. 1 \& 2]{Bjorck96}, where orthogonal projections, normal equations and the Moore-Penrose inverse are studied in detail.}:

\begin{enumerate}[(I)]
\item First, one can view \eqref{eq: modified f by right inverse} as ``projecting" an $r$-step scheme $\bm{f}^{\tau}$ consistent to $\bm{f}$ onto a scheme satisfying \eqref{eq: discrete correspondence b}, hence resulting in a conservative scheme implicitly. To better see this, suppose that the vector of conserved quantities $\bm{\psi}$ is independent of time explicitly, i.e. $\partial^{\tau}_t\bm{\psi}=\bm 0$. Then, satisfying condition \eqref{eq: discrete correspondence b} is equivalent to asking for the numerical scheme $\bm{f}^{\tau}_{\text{\scalebox{.7}{MN}}}$ to be in $\ker(\Lambda^{\tau})=\ran((\Lambda^{\tau})^\top)^\bot$. In other words, we seek to find a numerical scheme that is orthogonal to the row space of the discrete multiplier matrix $\Lambda^\tau$. Since the projection operator onto the row space of $\Lambda^\tau$ can be expressed as \cite[Eq. 1.2.29]{Bjorck96}
$$
P_{\ran((\Lambda^\tau)^\top)}  = (\Lambda^{\tau})^+\Lambda^\tau,
$$
the projection operator onto its orthogonal complement is then given by
$$
P_{\ker(\Lambda^{\tau})} = P_{\ran((\Lambda^\tau)^\top))^\perp} = I_{n\times n}-P_{(\Lambda^\tau)^\top} = I_{n\times n}-(\Lambda^{\tau})^+\Lambda^\tau,
$$ 
where $I_{n\times n}$ denotes the $n\times n$ identity matrix.
So in the case of time-independent conserved quantities, the MN-DMM scheme is equivalent to $\bm{f}^\tau_{\text{\scalebox{.7}{MN}}} = P_{\ker(\Lambda^\tau)} \bm{f}^{\tau}$, which automatically satisfies the discrete multiplier condition \eqref{eq: discrete correspondence b}.
Indeed, this follows by definition, since
$$\Lambda^\tau P_{\ker(\Lambda^\tau)} = \Lambda^\tau(I_{n\times n} - (\Lambda^{\tau})^+\Lambda^\tau) = \Lambda^\tau-\Lambda^\tau=  0_{m\times n}.$$


\item Alternatively, we can also take the point of view that any scheme satisfying \eqref{eq: discrete correspondence b} solves an undetermined linear system. Specifically, the general solution to \eqref{eq: discrete correspondence b} is given by $\bm f^\tau_0+\bm f^\tau_P$, where $\bm f^\tau_0 \in \ker(\Lambda^\tau)$ and $\bm f^\tau_P$ is any particular solution of \eqref{eq: discrete correspondence b}. Note that by direct substitution, one particular choice is provided by 
$$
\bm f^\tau_P := \left(\Lambda^{\tau}\right)^+ (-\partial^{\tau}_t\bm{\psi}).
$$ 
Since $\ker(\Lambda^\tau) = \ran(P_{\ker(\Lambda^{\tau})})$, then for any consistent $\bm {f}^\tau$, $\bm f^\tau_0=P_{\ker(\Lambda^{\tau})} \bm {f}^\tau$ and we arrive at the MN-DMM scheme:
$$
\bm f^\tau_0+\bm f^\tau_P =  \bm{f}^{\tau} -(\Lambda^{\tau})^+(\Lambda^{\tau}\bm{f}^\tau+\partial^{\tau}_t\bm{\psi})
= \bm{f}^{\tau}_{\text{\scalebox{.7}{MN}}} .
$$
Moreover, such a particular choice for $\bm f^\tau_P$ has the \emph{minimal $\ell^2$ norm} in the sense that for any consistent scheme $\bm f^\tau$, the MN-DMM scheme $\bm{f}^{\tau}_{\text{\scalebox{.7}{MN}}} $ is the closest scheme to $\bm f^\tau$ in the $\ell^2$ norm satisfying \eqref{eq: discrete correspondence b}:
\begin{equation}\label{eq: underdetermined minimal norm}
\bm{f}^{\tau}_{\scalebox{.7}{MN}} = \argmin_{\Lambda^\tau\tilde{\bm f}^\tau = -\partial_t^\tau\bm \psi} \norm{\bm f^\tau-\tilde{\bm f}^\tau}_{2}.
\end{equation}
Indeed, this follows from the observation that for any $\tilde{\bm f}^\tau$ satisfying \eqref{eq: discrete correspondence b},
\begin{align*}
\|\bm{f}^{\tau}-\tilde{\bm{f}}^\tau\|_2^2 &= \|(\bm f^\tau-\bm{f}^{\tau}_{\text{\scalebox{.7}{MN}}})+(\bm{f}^{\tau}_{\text{\scalebox{.7}{MN}}} -\tilde{\bm{f}}^{\tau})\|_2^2\\
&= \|\bm{f}^\tau-{\bm f}^{\tau}_{\text{\scalebox{.7}{MN}}}\|_2^2+\|\bm{f}^{\tau}_{\text{\scalebox{.7}{MN}}} -\tilde{\bm f}^{\tau}\|_2^2 
\geq \,\|\bm{f}^{\tau}-\bm f^{\tau}_{\text{\scalebox{.7}{MN}}}\|_2^2.
\end{align*}
Orthogonality in the second equality follows from $(\bm f^\tau_{\scalebox{.7}{MN}}-\tilde{\bm f}^\tau) \in \ker(\Lambda^\tau)$ by \eqref{eq: discrete correspondence b}. Moreover, $(\bm f^\tau-\bm f^\tau_{\scalebox{.7}{MN}}) \perp \ker(\Lambda^\tau)$, since for any $\bm f^\tau_0 \in \ker(\Lambda^\tau)$,
\begin{align*}
    (\bm f^\tau-\bm f^{\tau}_{\text{\scalebox{.7}{MN}}})\cdot\bm f^\tau_0 &= (\Lambda^\tau)^+(\Lambda^\tau \bm f^\tau+\partial_t^\tau\bm \psi)\cdot \bm f^\tau_0 \\ &= (\Lambda^\tau(\Lambda^\tau)^\top)^{-1}(\Lambda^\tau \bm f^\tau+\partial_t^\tau\bm \psi)\cdot(\Lambda^\tau\bm f^\tau_0)= \bm 0.
\end{align*}
\end{enumerate}

\subsection{Theory} \label{sec:theory}

Before we prove that \eqref{eq: modified f by right inverse} defines a consistent conservative scheme, we will need show that the pseudoinverses $(\Lambda^{\tau})^+$ is well-defined and are uniformly bounded for small enough $\tau$. 

Recall the hypotheses of assumptions \ref{assum: nb of conserved quantities} and \ref{assum: consistency} introduced in \Cref{sec: Review of DMM}:
\begin{itemize}
    \item $\Lambda(t,\bm{x})$ has full row rank in $I\times U$, \label{hyp:rank}
    \item $\Lambda^{\tau}=\Lambda +\mathcal{O}(\tau^q)$\label{hyp:consistency}.
\end{itemize}
These hypotheses imply that for any $\bm{x}\in C^{p+q}(I\rightarrow\mathbb{R}^n)$, the rows of $\Lambda^{\tau}$ are linearly independent for $\tau$ small enough.
\begin{lemma}\label{lem: rank of Lambda tau} Under assumptions \ref{assum: nb of conserved quantities} and \ref{assum: consistency}, for any $\bm{x}\in C^{p+q}(I\rightarrow\mathbb{R}^n)$, there exists $\tau_0>0$ such that $\Lambda^{\tau}$ has for full row rank whenever $\tau<\tau_0$.
\end{lemma}
\begin{proof}
We argue by contradiction. Suppose that for some $\bm{x}\in C^{p+q}\left(I\rightarrow U\right)$, there exists a sequence of \emph{unit} vectors $(\bm{u}_k)_{k\in\mathbb{N}}\in \mathbb{R}^n$ and a sequence of parameters $(\tau_k)_{k\in\mathbb{N}}\in\mathbb{R}$ with $\tau_{k}\longrightarrow 0$  as $k\rightarrow \infty$, such that $(\Lambda^{\tau_k})^{\top}\bm{u}_k=\bm{0}$ for all $k\in\mathbb{N}$. Then, since the unit sphere $S^{n-1}=\{\bm{x}\in\mathbb{R}^n \vert \|\bm{x\|}=1\}$ is compact in $\mathbb{R}^m$, it follows from the hypothesis that $\Lambda$ has full row rank and the extreme value theorem that a positive lower bound  $\alpha:= \min_{\|\bm{v}\|=1}\|\Lambda^{\top}\bm{v}\|$ is achieved. Since $\Lambda^{\tau}$ is consistent with $\Lambda$, we therefore obtain the contradiction that
$$
0 < \alpha \leq \|\Lambda^{\top}\bm{u}_k\|\leq \|(\Lambda^{\top}-(\Lambda^{\tau_k})^{\top})\bm{u}_k\|\leq  \|\Lambda-\Lambda^{\tau_k}\|_{m\times n} \longrightarrow 0,\quad\text{as }k\rightarrow \infty.
$$\end{proof}

As a consequence, the Moore-Penrose right inverse of $\Lambda^{\tau}$ is well-defined and given by \cite[Eq. 1.2.27]{Bjorck96}
\begin{equation}\label{eq: moore-penrose explicit rep with AAT}
(\Lambda^{\tau})^+ = (\Lambda^{\tau})^{\top}(\Lambda^{\tau}(\Lambda^{\tau})^{\top})^{-1},\qquad \tau<\tau_0.\end{equation} 
Moreover, we see that $\bm{f}^{\tau}_{\text{\scalebox{.7}{MN}}}$ is a well-defined $r$-step function in the sense of \Cref{sec: Review of DMM}. Indeed, it is clear from $\smash{(\Lambda^{\tau}(\Lambda^{\tau})^{\top})^{-1} = (\text{det}(\Lambda^{\tau}(\Lambda^{\tau})^{\top}))^{-1}\text{adj}(\Lambda^{\tau}(\Lambda^{\tau})^{\top})}$  that the matrix $(\Lambda^{\tau})^+$ is of class $C^{l}$ whenever $\Lambda^{\tau}\in C^{l}$, $l\in\mathbb{N}$.

Thanks to the next lemma, the expression \eqref{eq: moore-penrose explicit rep with AAT} is also useful in proving that for any $\bm{x}\in C^{p+q}\left(I\rightarrow U\right)$, the parametrized family $\{(\Lambda^{\tau})^+\}_{\tau}$ is eventually uniformly bounded in $\tau$ as $\tau\rightarrow 0$.

\begin{lemma}\label{lem: existence and boundedness of right-inverse}
	Under assumptions \ref{assum: nb of conserved quantities} and \ref{assum: consistency}, for any $\bm{x}\in C^{p+q}(I\rightarrow U)$, there exists a parameter $\tau^0>0$ and a constant $C(\Lambda,\bm{x})>0$ independent of $\tau$ such that whenever $0<\tau<\tau^0$, the inverse $(\Lambda^{\tau}(\Lambda^{\tau})^{\top})^{-1}$ exists and satisfies
	\begin{equation}\label{eq: estimate of lemma boundedness}
	\|(\Lambda^{\tau}(\Lambda^{\tau})^{\top})^{-1}\|_{m\times m} \leq C(\Lambda,\bm{x}).
	\end{equation}
\end{lemma}	
\begin{proof}
We see from combining the estimates
	\begin{align*}
	\|\Lambda^{\tau}(\Lambda^{\tau})^{\top}-\Lambda\Lambda^{\top}\|_{m\times m}
	&\leq \|\Lambda^{\tau}(\Lambda^{\tau})^{\top}-\Lambda^{\tau}\Lambda^{\top}\|_{m\times m}+\,\|\Lambda^{\tau}\Lambda^{\top}-\Lambda\Lambda^{\top}\|_{m\times m} \nonumber\\
	&\leq (\,\,\norm{\Lambda^{\tau}}_{m \times n}+\|\Lambda^{\top}\|_{n\times m})
	\,\|\Lambda^{\tau}-\Lambda\|_{m\times n} 
	\end{align*}
and
\begin{align}
	\label{eq: bound on Lambda tau}
	\|\Lambda^{\tau}\|_{m\times n} &\leq \|\Lambda^{\tau}-\Lambda\|_{m\times n}+\,\|\Lambda\|_{m\times n},
	\end{align}
that $\Lambda^{\tau}(\Lambda^{\tau})^{\top}=\Lambda\Lambda^{\top}+\mathcal{O}(\tau^q)$. Therefore, the desired conclusion follows from a well-known result concerning perturbation of regular matrices \cite[Thm. 1.5]{quarteroni2010numerical}, which in the current setting states that if there exists $\tau_0>0$ such that
	\begin{equation*}
	\|\Lambda^{\tau}(\Lambda^{\tau})^{\top}-\Lambda\Lambda^{\top}\|_{m\times m} < \|(\Lambda\Lambda^{\top})^{-1}\|_{m\times m}^{-1},
	\end{equation*}
	then the matrices $\Lambda^{\tau}\left(\Lambda^{\tau}\right)^{\top}$ are invertible on the interval $(0,\tau_0)$ and bounded by a constant independent of $\tau$.
\end{proof}
\begin{corollary}\label{cor: moore penrose inverse existence}
Under assumptions \ref{assum: nb of conserved quantities} and \ref{assum: consistency}, for any $\bm{x}\in C^{p+q}(I\rightarrow U)$, there exists a parameter $\tau^0>0$ and a constant $C(\Lambda,\bm{x})>0$ independent of $\tau$ such that
\begin{equation*}
    	\|(\Lambda^{\tau})^+\|_{n\times m} \leq C(\Lambda,\bm{x}),\qquad \tau<\tau_0.
\end{equation*}
\end{corollary}
\begin{proof}
Based on \eqref{eq: moore-penrose explicit rep with AAT}, we find that for any $\bm{x}\in C^{p+q}(I\rightarrow U)$, we have
$$
	\|(\Lambda^{\tau})^+\|_{n\times m}=\|(\Lambda^\tau)^{\top}(\Lambda^{\tau}(\Lambda^{\tau})^{\top})^{-1}\|_{n\times m}
	\leq\|(\Lambda^\tau)^{\top}\|_{m\times m}\|(\Lambda^{\tau}(\Lambda^{\tau})^{\top})^{-1}\|_{m\times m}
$$
for $\tau$ small enough. In particular, $\tau_0$ can be chosen as in the proof of \Cref{lem: existence and boundedness of right-inverse}.
\end{proof}

The next theorem shows that upon satisfying the discrete multiplier condition \eqref{eq: discrete correspondence a}, which as previously mentioned can be resolved by discrete chain rules \cite{wan2017conservative}, the MN-DMM indeed leads to a conservative scheme.

\begin{theorem}\label{thm: A-DMM}
	Under assumptions \ref{assum: nb of conserved quantities} and \ref{assum: consistency}, suppose that the discrete quantities of \Cref{sec: Review of DMM} satisfy the compatibility condition \eqref{eq: discrete correspondence a} for all $(t^k,\bm{x}^{k+1},...,\bm{x}^{k-r+1})\in I\times U^{r+1}$ such that
\begin{equation}\label{FMN r-step method}
    \bm{F}_{\text{\emph{\scalebox{.7}{MN}}}}^{\tau}(t^k,\bm{x}^{k+1},...,\bm{x}^{k-r+1})=\bm{0},
\end{equation}
    where 
    $$\bm{F}_{\text{\emph{\scalebox{.7}{MN}}}}^{\tau}:=D^{\tau}_t\bm{x}-\bm{f}^{\tau}_{\text{\emph{\scalebox{.7}{MN}}}}.$$
    Then, the $r$-step method defined by \eqref{FMN r-step method}
	is conservative in $\bm{\psi}^{\tau}$. Moreover, it is consistent of at least order $q$ to the function $\bm{F}$ defined in \eqref{eq: quasilinear ODE}, and for any $\bm{x}\in C^{p+q}\left(I\rightarrow\mathbb{R}^n\right)$ the discrete quantities satisfy
	\begin{subequations}
	\begin{align}
	\Lambda^{\tau}D^{\tau}_t\bm{x} - D^{\tau}_t\bm{\psi} - \partial^{\tau}_t\bm{\psi} &= \mathcal{O}(\tau^q),\\
	\Lambda^{\tau}\bm{f}_{\text{\emph{\scalebox{.7}{MN}}}}^{\tau} + \partial^{\tau}_t \bm{\psi} &= \mathcal{O}(\tau^q).
	\end{align}
	\end{subequations}
\end{theorem}
\begin{proof}
 Our goal is to resort to \Cref{thm: conservative scheme theory}. Two ingredients are required. 
	
	First, we need to confirm that the discrete function $\bm{f}_{\text{\scalebox{.7}{MN}}}^{\tau}$ verifies the second discrete multiplier condition \eqref{eq: discrete correspondence b}. This holds by construction. Since by definition $\Lambda^{\tau}\left(\Lambda^{\tau}\right)^{+}=I_{m\times m}$, multiplying both sides of \eqref{eq: modified f by right inverse} by $\Lambda^{\tau}$ immediately yields
	\begin{equation*}
	\Lambda^{\tau}\bm{f}_{\text{\scalebox{.7}{MN}}}^{\tau} = \Lambda^{\tau}\bm{f}^{\tau}- \Lambda^{\tau}(\Lambda^{\tau})^{+}(\Lambda^{\tau}\bm{f}^{\tau} + \partial^{\tau}_t\bm \psi) = -\partial^{\tau}_t\bm \psi.
	\end{equation*}
	
	Second, we need to show that $\bm{f}_{\text{\scalebox{.7}{MN}}}^{\tau}= \bm{f}+\mathcal{O}(\tau^q)$. Since the triangle inequality yields
	\begin{equation}\label{eq: consistency estimate ADMM}
	\|\bm{f}-\bm{f}_{\text{\emph{\scalebox{.7}{MN}}}}^{\tau}\|
	\leq  \|\bm{f}-\bm{f}^{\tau}\|+ \|\left(\Lambda^{\tau}\right)^{+}\|_{n\times m}\|\Lambda^{\tau}\bm{f}^{\tau} + \partial^{\tau}_t\psi\|,
	\end{equation}
	it follows from \Cref{cor: moore penrose inverse existence} that we only need to verify that $\Lambda^{\tau}\bm{f}^{\tau} + \partial^{\tau}_t\psi=\mathcal{O}\left(\tau^q\right)$.
	
	Consider the estimate
	\begin{align}\label{eq: bound on lambda f plus partial phi}
	\|\Lambda^{\tau}\bm{f}^{\tau} + \partial^{\tau}_t\bm{\psi}\| 
	&\leq  \|\Lambda^{\tau}\bm{f}^{\tau}-\Lambda^{\tau}\bm{f}\| + \|\Lambda^{\tau}\bm{f}-\Lambda\bm{f}\| + \|\Lambda\bm{f}+\partial^{\tau}_t\bm{\psi}\|\nonumber \\
	&\leq\|\Lambda^{\tau}\|_{m\times n}\|\bm{f}^{\tau}-\bm{f}\| +\|\Lambda^{\tau}-\Lambda\|_{m \times n}\|\bm{f}\|+ \|\Lambda\bm{f}+\partial^{\tau}_t\bm{\psi}\|.
	\end{align}
	The key observation is that since $\Lambda$ is a conservation law multiplier of $\bm{F}$ associated to $\bm{\psi}$ by hypothesis, it satisfies the correspondence identity \eqref{eq: continuous correspondence identity b}, i.e. $\Lambda\bm{f}=-\partial_t\bm{\psi}$. Introducing $\partial_t\bm{\psi}$ in the last term of \eqref{eq: bound on lambda f plus partial phi} yields
	\begin{equation}\label{eq: psi is psi tau}
	\|\Lambda\bm{f}+\partial^{\tau}_t\bm{\psi}\|= \|\partial_t\bm{\psi}-\partial^{\tau}_t\bm{\psi}\|.
	\end{equation}

	Upon inserting \eqref{eq: psi is psi tau} in \eqref{eq: bound on lambda f plus partial phi}, then \eqref{eq: bound on lambda f plus partial phi} in \eqref{eq: consistency estimate ADMM}, the proof follows by consistency of $\bm{f}^{\tau}$, $\Lambda^{\tau}$ and $\partial_t^{\tau}\bm{\psi}$ to $\bm{f}$, $\Lambda$ and $\partial_t\bm{\psi}$, respectively.
\end{proof}	

\section{Practical Implementations}\label{sec: Practical Implementation}
As the MN-DMM scheme \eqref{FMN r-step method} is implicitly defined, we turn to an iterative fixed point algorithm in order to converge to the desired conservative scheme and simultaneously solve the associated nonlinear equations. 
Following \cite{wan2017conservative}, we will focus on one-step conservative methods constructed by using divided differences for 
\begin{subequations}
\begin{align}
\bm{\psi}^{\tau}(t^k,\bm{x}^k)&:=\bm{\psi}(t^k,\bm{x}^k),\label{eq: 1st order psi}\\
D^{\tau}_{t}\bm{x}(t^k,\bm{x}^{k+1},\bm{x}^k)&:=\frac{\bm{x}^{k+1}-\bm{x}^k}{t^{k+1}-t^k}, \label{eq: 1st order Dt} \\
D^{\tau}_t\bm{\psi}(t^k,\bm{x}^{k+1},\bm{x}^k)&:=\frac{\bm{\psi}(t^k,\bm{x}^{k+1})-\bm{\psi}(t^k,\bm{x}^{k})}{t^{k+1}-t^k},\label{eq: 1st order Dtpsi}
\end{align}
\end{subequations}
which are used throughout in the numerical results presented in \Cref{sec: Numerical results}.

It is clear that these discrete quantities are consistent single-step functions of at least first-order to their continuous counterpart. Conveniently, constant compatibility of $D^{\tau}_t\bm{\psi}$ with $\bm{\psi}^{\tau}$ is immediate. We refer to \cite{wan2017conservative} for the derivation of a single-step function $\Lambda^{\tau}$ using discrete chain rules that satisfy condition \eqref{eq: discrete correspondence a}. 

Some higher-order multi-step DMM schemes were constructed in \cite{wan2018stability}. Also note that first-order symmetric schemes can turn out to be high-order as well \cite[Chapter II.3, Theorem 3.2]{hair06Ay}, which was studied in the conservative DMM schemes for many-body problems \cite{manybody2022}, vortex blob methods \cite{vortex2022}, and Hamiltonian Monte Carlo methods \cite{chmc22}.

\subsection{Analytic expressions of MN-DMM for small number of conserved quantities}
For a small number of conserved quantities, it is in fact analytically tractable to write out the expressions of MN-DMM given by \eqref{eq: modified f by right inverse}. For ease of future reference, we write out the explicit MN-DMM schemes for preserving one and two conserved quantities, i.e. $m=1$ and $2$. Specifically, the case $m=1$ leads to a simple way to enable conservation for an arbitrary consistent scheme and in a gradient-free manner. For instance, this could be highly relevant to physical systems where energy conservation is important, such as for Hamiltonian systems. 

\subsubsection{Analytic expression for m=1}
\label{sec: analytic m=1}

For a single scalar conserved quantity, the discrete multiplier matrix $\Lambda^\tau\in M_{1\times n}(\mathbb{R})$ is the row vector

\[
\Lambda^\tau(t^k,\bm x^{k+1}, \bm x^k) := \frac{\Delta \psi}{\Delta \bm x}^{\top}(t^k,\bm x^{k+1}, \bm x^k),
\] where $\dfrac{\Delta \psi}{\Delta \bm x}$ denotes the column vector of partial divided differences of $\psi$ with respect to $\bm x$ for a specific permutation of $S_{n+1}$ satisfying the discrete chain rule \eqref{eq: discrete correspondence a}\footnote{Details on divided difference calculus and explicit formulas for $ \frac{\Delta \psi}{\Delta \bm x}$ are in Appendix B of \cite{wan2017conservative}.}. Since $\Lambda^\tau {(\Lambda^\tau)}^{\top} = \norm{\dfrac{\Delta \psi}{\Delta \bm x}}_2^2$ is a scalar quantity in this case, we see that the MN-DMM scheme of \eqref{eq: modified f by right inverse} for $m=1$ is given by
\begin{equation}
    \label{eq: MN-DMM m=1}
    \bm f_{\text{\scalebox{.7}{MN}}}^\tau:=\bm f^\tau - \frac{1}{\norm{\frac{\Delta \psi}{\Delta \bm x}}_2^2}\left(\frac{\Delta \psi}{\Delta \bm x}^{\top}\bm f^\tau+\partial_t^\tau \psi\right)\frac{\Delta \psi}{\Delta \bm x},
\end{equation}where we have suppressed the arguments $(t^k,\bm x^{k+1}, \bm x^k)$ for clarity. As seen in \Cref{sec: Minimal l2 Norm Discrete Multiplier Method}, for time-independent $\psi$, \eqref{eq: MN-DMM m=1} can be viewed as subtracting off the projection of the scheme $\bm f^\tau$ onto the orthogonal complement of the discrete multiplier's kernel. Since the kernel is in this case the multi-dimensional plane perpendicular to the vector $\frac{\Delta \psi}{\Delta \bm x}$, its orthogonal complement is simply the span of the latter, and the resulting scheme reads
\begin{equation}\label{eq: MN-DMM m=1 time independent}
    \bm f_{\text{\scalebox{.7}{MN}}}^\tau:=\bm f^\tau - \alpha \frac{\Delta \psi}{\Delta \bm x}, \qquad\text{ where } \alpha:=\frac{1}{\norm{\frac{\Delta \psi}{\Delta \bm x}}_2^2}\frac{\Delta \psi}{\Delta \bm x}^{\top}\bm f^\tau.
\end{equation}
In other words, $\alpha \frac{\Delta \psi}{\Delta \bm x}$ is the scalar projection of $\bm{f}^{\tau}$ onto $\frac{\Delta \psi}{\Delta \bm x}$ and $\bm f_{\text{\scalebox{.7}{MN}}}^\tau$ is the vector projection of $\bm{f}^{\tau}$ onto the kernel of $\Lambda^\tau=\frac{\Delta \psi}{\Delta \bm x}^{\top}$, as discussed in \Cref{sec: Minimal l2 Norm Discrete Multiplier Method}. The expression in \eqref{eq: MN-DMM m=1 time independent} conveys how MN-DMM schemes arise as $\ell^2$-projections. It also demonstrates the ease with which a consistent scheme can be amended to a consistent conservative one.

\subsubsection{Analytic expression for m=2}
For two conserved quantities, the discrete multiplier matrix $\Lambda^\tau \in M_{2\times n}(\mathbb{R})$ is given by

\[
\Lambda^\tau(t^k,\bm x^{k+1}, \bm x^k) := \begin{pmatrix}\dfrac{\Delta \psi_1}{\Delta \bm x}^{\top}(t^k,\bm x^{k+1}, \bm x^k) \\
\dfrac{\Delta \psi_2}{\Delta \bm x}^{\top}(t^k,\bm x^{k+1}, \bm x^k) \end{pmatrix},
\] where $\dfrac{\Delta \psi_i}{\Delta \bm x}$ again denotes the column vector of partial divided differences of $\psi_i$ with respect to $\bm x$, similar to the $m=1$ case. With $\Lambda^\tau {(\Lambda^\tau)}^{\top}$ now being a $2\times 2$ matrix, the MN-DMM scheme of \eqref{eq: modified f by right inverse} for $m=2$ takes the explicit form

\begin{align}
    \label{eq: MN-DMM m=2}
    \bm f_{\text{\scalebox{.7}{MN}}}^\tau:= \bm f^\tau - \frac{\begin{pmatrix} \frac{\Delta \psi_1}{\Delta \bm x} & \frac{\Delta \psi_2}{\Delta \bm x} \end{pmatrix}}{\det(\Lambda^\tau {(\Lambda^\tau)}^{\top})}\begin{pmatrix}\norm{\frac{\Delta \psi_2}{\Delta \bm x}}_2^2 & -\frac{\Delta \psi_1}{\Delta \bm x}^{\top}\frac{\Delta \psi_2}{\Delta \bm x}\\
-\frac{\Delta \psi_2}{\Delta \bm x}^{\top}\frac{\Delta \psi_1}{\Delta \bm x} & \norm{\frac{\Delta \psi_1}{\Delta \bm x}}_2^2\end{pmatrix}\begin{pmatrix}\frac{\Delta \psi_1}{\Delta \bm x}^{\top}\bm f^\tau+ \partial_t^\tau \psi_1\\ \frac{\Delta \psi_2}{\Delta \bm x}^{\top}\bm f^\tau+ \partial_t^\tau \psi_2\end{pmatrix},
\end{align}
where $\det(\Lambda^\tau {(\Lambda^\tau)}^{\top}) = \norm{\frac{\Delta \psi_1}{\Delta \bm x}}_2^2\norm{\frac{\Delta \psi_2}{\Delta \bm x}}_2^2- \left(\frac{\Delta \psi_2}{\Delta \bm x}^{\top}\frac{\Delta \psi_1}{\Delta \bm x}\right)^2$.

In principle, MN-DMM schemes for other small $m$ values can also be written out analytically. However, for practical implementations involving $m>2$, we instead refer to Section \ref{sec: MN-DMM Algorithm} to \ref{sec: SVD MN-DMM} for more general algorithms that implicitly construct conservative schemes without resorting to analytic computations.

\subsection{Direct MN-DMM Algorithm}\label{sec: MN-DMM Algorithm}
For an arbitrary number $m$ of conserved quantities, we now present a fixed-point iteration algorithm associated with the scheme \eqref{FMN r-step method} introduced in \Cref{thm: A-DMM}, where consistency and conservative properties were shown. 
Before we compare different ways of computing the pseudoinverse expression involved in \eqref{FMN r-step method}, let us describe how to solve the implicit scheme
$$\bm{0} = \bm{F}_{\text{\emph{\scalebox{.7}{MN}}}}^{\tau}(t^k,\bm{x}^{k+1},\bm{x}^{k})= D_t^\tau \bm x-\bm{f}^{\tau}_{\text{\emph{\scalebox{.7}{MN}}}}(t^k,\bm{x}^{k+1},\bm{x}^{k}).$$ More explicitly, this is equivalent to the equations
\begin{equation}\label{eq: MN DMM explicit}
\bm x^{k+1} =\bm x^k+(t^{k+1}-t^k)\bm{f}^{\tau}_{\text{\emph{\scalebox{.7}{MN}}}}(t^k,\bm{x}^{k+1},\bm{x}^{k}).
\end{equation}

We will employed a fixed-point iteration to solve for $\bm x^{k+1}$ in \eqref{eq: MN DMM explicit}, which can also be viewed as a predictor-corrector method. For brevity and clarity, we shall consider an uniform  time step\footnote{Similar results can be derived with variable time steps by replacing $\tau$ with $\tau_k := t^{k+1}-t^k$ and ensuring $\tau := \sup_{k}(t^{k+1}-t^k)$ satisfies the contraction criteria in the fixed point iteration.} $\tau = t^{k+1}-t^k$ for all $k$, and denote the unknown vector as $\bm x:=\bm x^{k+1}$ and the fixed vector as $\bm y:=\bm x^{k}$. To bootstrap the fixed point iteration, we first compute an initial guess $\bm x^{(0)}=\phi(t^k,\bm y)$ using any sufficiently accurate explicit time-stepping scheme $\phi:I\times U\rightarrow U$ depending on $\bm{f}$, $t^k$ and $\bm y$, such as explicit Runge-Kutta schemes. From this initial guess or predictor, subsequent iterates $\bm x^{(i)}$ are then improved or corrected using the implicit MN-DMM scheme by iterating the fixed point iteration of \eqref{eq: MN DMM explicit} given by
\begin{equation*}
    \bm{x}^{(i)} := \bm{y} + \tau \bm{f}_{\text{\emph{\scalebox{.7}{MN}}}}^{\tau}(t^k,\bm{x}^{(i-1)},\bm{y})
\end{equation*}
until a desired tolerance $\delta$ is reached. More explicitly, short-handing the notations
\begin{align*}
A(\bm x)&:=\Lambda^\tau(t^k,\bm x, \bm{y}),&&\mbox{(Discrete multiplier matrix)} \\
\bm{s}(\bm x)&:=\bm{f}^{\tau}(t^k,\bm x, \bm{y}), && \mbox{(Discrete source term)}\\
\bm{r}(\bm x)&:=A(\bm x)\bm{s}(\bm x)+\partial_t^\tau \bm{\psi}(t^k,\bm x, \bm{y}), &&\mbox{(Residual of \eqref{eq: discrete correspondence b})}
\end{align*}
and using the absolute error of the conserved quantities as the tolerance criteria, we arrive at the \emph{Direct MN-DMM Algorithm}, or \emph{MN-DMM Algorithm}:

\begin{algorithm}[h!]
	\caption{Direct MN-DMM}\label{alg: MN-DMM}
	\begin{algorithmic}[1]
		\State $\bm{x}^{(0)}\gets \phi(t^k,\bm y)$
		\Repeat $~i=1,2,\dots$
		\State $\bm{x}^{(i)} \gets \bm{y} +\tau\,\left(\bm{s}(\bm x^{(i-1)}) - A^+(\bm{x}^{(i-1)})\,\bm{r}(\bm{x}^{(i-1)})\right)$
	\Until{$\abs{\bm\psi(\bm{x}^{(i)})-\bm\psi(\bm{x}^0)}<\delta$}
		\State \Return $\bm{x}^{(i)}$
	\end{algorithmic}
\end{algorithm}
\vskip -2mm
A Banach fixed point argument shows that \Cref{alg: MN-DMM} converges. 
\begin{theorem} \label{thm:FPIconv}
If for sufficiently small $\tau$, the collection of functions $\{\bm{s}$, $\bm{A}^+$, $\bm{r}\}_{\tau}$ are locally Lipschitz continuous with Lipschitz constants independent of $\tau$, then under the hypotheses of assumptions \ref{assum: nb of conserved quantities} and \ref{assum: consistency}, there exists $\tau_*>0$ such that \Cref{alg: MN-DMM} converges whenever $\tau<\tau_*$.
\end{theorem}
\begin{proof}
Denote $\bm{G}^{\tau}(\bm{z}) := \bm{y} + \tau\,\bm{F}^\tau(\bm{z})$ and $\bm{F}^{\tau}(\bm{z}):=\bm{s}(\bm z) - A^+(\bm{z})\,\bm{r}(\bm{z})$. Then the above algorithm is equivalent to the fixed point iteration 
$$
\bm{x}^{(i+1)} = \bm{G}^{\tau}(\bm{x}^{(i)}),\qquad\bm{x}^{(0)}:=\phi(t^k,\bm{y}).
$$

By continuity, it follows by \Cref{lem: existence and boundedness of right-inverse} and consistency of $\bm{s}$ and $\bm{r}$ to their continuous counterpart that there exists $\tau_0>0$ and an open ball $\mathcal{B}\subset\mathbb{R}^n$ of radius $\epsilon>0$ centered at $\bm{y}$ over which the restrictions of the discrete functions are Lipschitz continuous and $M:= \sup_{\tau<\tau_0}\sup_{\bm{z}\in \mathcal{B}}\,\|\bm{F}^{\tau}(\bm{z})\| <\infty
$. In particular, $\bm{G}^{\tau}(\mathcal{B})\subset\mathcal{B}$ for $\tau < \min\{\tau_0,\epsilon/M\}$. In fact, the Lipschitz continuity hypothesis guarantees that
\begin{equation}
    \|\bm{G}^{\tau}(\bm{z}_1)-\bm{G}^{\tau}(\bm{z}_2)\| = \tau \| \bm{F}^{\tau}(\bm{z}_1)-\bm{F}^{\tau}(\bm{z}_2) \| \leq \tau\,L \|\bm{z}_1-\bm{z}_1\|
\end{equation}
where $L>0$ is the Lipschitz constant of $\bm{F}^{\tau}$ over $\mathcal{B}$. We conclude that $\bm{G}^{\tau}:\mathcal{B}\rightarrow \mathcal{B}$ is a contraction for $\tau<\tau_*:=\min\{\tau_0,\epsilon/M,1/L\}$, and thus \Cref{alg: MN-DMM} converges by the Banach fixed point theorem.
\end{proof}
\vskip -2mm
The main drawback of the Direct MN-DMM Algorithm is the need to compute analytically an inverse matrix within the pseudoinverse of $A(\bm x)$. This can be alleviated by introducing auxiliary variables, as we discuss next.
\vskip -2mm
\subsection{Mixed MN-DMM}	\label{sec: Mixed MN-DMM} 
In order to solve \eqref{eq: MN DMM explicit} without having to invert $\Lambda^{\tau}\left(\Lambda^{\tau}\right)^{\top}$ explicitly for the computation of $\bm{f}^{\tau}_{\text{\emph{\scalebox{.7}{MN}}}}$, one option is to consider the mixed formulation of \ref{eq: MN DMM explicit},
\begin{subequations}
\begin{align}
D_t^\tau \bm{x} + {\Lambda^\tau}^{\top}\bm{g} &= \bm{f}^\tau, \label{eq: mixed1}\\ 
\Lambda^\tau{\Lambda^\tau}^{\top} \bm{g} &= {\Lambda^\tau}\bm{f}^\tau+\partial_t^\tau \bm{\psi}, \label{eq: mixed2}
\end{align}
\end{subequations}
where the matrix inversion is replaced with solving the linear system \eqref{eq: mixed2}.  Denoting
$
B(\bm x):=A(\bm x)A(\bm x)^{\top}
$, equations \eqref{eq: mixed1} and \eqref{eq: mixed2} are equivalent to 
\begin{subequations}
\begin{align}
\bm x &= \bm{y} + \tau(\bm{s}(\bm x)-A(\bm x)^{\top}\bm g), \label{eq: simp mixed1}\\ 
B(\bm x)\bm{g} &= \bm{r}(\bm x). \label{eq: simp mixed2}
\end{align}
\end{subequations}

To solve \eqref{eq: simp mixed1} and \eqref{eq: simp mixed2}, we again propose a fixed point iteration type algorithm, which we referred to as the \emph{Mixed MN-DMM Algorithm}:

\begin{algorithm}
	\caption{Mixed MN-DMM}\label{alg: Mixed A-DMM}
	\begin{algorithmic}[1]
		\State $\bm{x}^{(0)}\gets \phi(t^k,\bm y)$
		\Repeat $~i=1,2,\dots$
		\State $\bm{g} \gets \text{Solve}\left(B(\bm x^{(i-1)})\bm{g} = \bm{r}(\bm x^{(i-1)}) \right)$
		\State $\bm{x}^{(i)} \gets \bm{y} +\tau(\bm{s}(\bm x^{(i-1)})-A(\bm x^{(i-1)})^{\top}\bm g)$
		\Until{$\abs{\bm\psi(\bm x^{(i)})-\bm\psi(\bm{x}^0)}<\delta$}
		\State \Return $\bm x^{(i)}$
	\end{algorithmic}
\end{algorithm}
\vskip -2mm
Notice that for any accurate enough initial guess $\bm{x}^{(0)}$, standard arguments for perturbation of matrices that we have previously used in \Cref{lem: existence and boundedness of right-inverse} guarantees that $B$ will be invertible for sufficiently small $\tau$. 
In other words, \Cref{alg: Mixed A-DMM} is iteratively solving normal equations of the second kind 
\[
B(\bm{x}) \bm{g} = \bm{r},\qquad\qquad A^{\top}\bm{g} = \bm{f}^{\tau}-\bm{f}_{\text{\scalebox{.7}{MN}}}^{\tau},
\]
associated with the \emph{under}determined minimization problem \eqref{eq: underdetermined minimal norm}
, see for example \cite[Eq. 1.1.20]{Bjorck96}. Moreover, in line 4 of \Cref{alg: Mixed A-DMM}, we have the freedom to choose any state of the art linear solver for this type of equation. However, it is well-known that forming $B(\bm x)=A(\bm x)A(\bm x)^{\top}$ explicitly may lead to loss of accuracy and large condition numbers. Taking this possibility into account, we propose next using matrix decomposition techniques that are better suited to tackle such instances.

\subsection{Mixed MN-DMM using Singular Value Decomposition} \label{sec: SVD MN-DMM}

As discussed, the matrix $B$ can be ill-conditioned in practice, and we will see this in some numerical examples of \Cref{sec: Numerical results}. Appealing to the Singular Value Decomposition (SVD) \cite{trefethenBau97},
\[
A = U\Sigma V^{\top}
\]
can alleviated this issue\footnote{Indeed, QR decomposition is another possibility as well.}, though at additional costs of computing such decomposition. Recall here that $U\in M_{m\times m}(\mathbb{R})$ and $V\in M_{n\times n}(\mathbb{R})$ are orthogonal matrices and the non-zero block of
$\Sigma = \begin{pmatrix}
\Sigma_m & 0_{m\times (n-m)}
\end{pmatrix}
$
is the diagonal matrix $\Sigma_m := \text{diag}(\sigma_1,...,\sigma_m)$, where $\sigma_1\geq ...\geq \sigma_m$ are the real eigenvalues of $B$. Thus, the multiplication of the right Moore-Penrose inverse on  $\bm{f}_{\text{\emph{\scalebox{.7}{MN}}}}^{\tau}$ can be computed using 
\[
A^+ = V \Sigma^{+}
U^{\top},\qquad\Sigma^{+}:=\begin{pmatrix}
\Sigma_m^{-1}\\
0_{(n-m)\times m}
\end{pmatrix},
\] which can be done in a sequential manner involving only matrix--vector products. We refer this approach as the \emph{Mixed MN-DMM Algorithm using SVD}: 

\begin{algorithm}
	\caption{Mixed MN-DMM using SVD}\label{alg: SVD MN-DMM}
	\begin{algorithmic}[1]
		\State $\bm{x}^{(0)}\gets \phi(t^k,\bm y)$
		\Repeat $~i=1,2,\dots$
		\State $[U,\Sigma, V] \gets \text{SVD}(A(\bm{x}^{(i-1)}))$
		\State $\bm{a} \gets U^{\top}\bm{r}(\bm{x}^{(i-1)})$
		\State $\bm{b} \gets \Sigma^{+}\bm{a}$
		\State $\bm{x}^{(i)} \gets \bm{y} +\tau\,\left(\bm{s}(\bm{x}^{(i-1)}) - V\bm{b}\right)$
	\Until{$\abs{\bm\psi(\bm{x}^{(i)})-\bm\psi(\bm{x}^0)}<\delta$}
		\State \Return $\bm{x}^{(i)}$
	\end{algorithmic}
\end{algorithm}
\vskip -2mm
The main advantage of this approach is that the product $B(\bm x)=A(\bm x)A(\bm x)^{\top}$ does not need to be assembled at each iteration, thus potentially improving the accuracy of the solution for poorly conditioned problems. However, the main drawback is that computing the SVD decomposition of $A(\bm x)\in M_{m\times n}(\mathbb{R})$ at each iteration requires additional costs. Nevertheless, \Cref{alg: SVD MN-DMM} can yield more accurate numerical solutions when $A$ is poorly conditioned, which opens the possibility to future improvements along this direction.

\section{Numerical results}\label{sec: Numerical results}

With the theoretical results now established and practical implementation discussed, we now present several numerical examples to illustrate the MN-DMM approach and its two variants. The examples were chosen from a wide variety of physical problems, such as biological systems, chaotic systems, classical mechanics, fluid dynamics and geodesic flows. Moreover, they are roughly ordered at increasing difficulty in deriving analytic conservative schemes using the original DMM approach. In contrast, the MN-DMM approach only requires knowledge of the divided difference expressions within the discrete multiplier matrix to construct the conservative schemes, which can readily be systematized using modern computer algebra packages.

For the following examples, we have chosen to compare the MN-DMM method with two traditional methods, namely the standard 4th-order Runge-Kutta method and the 2nd-order symplectic Implicit Midpoint method. While these choices do not form an exhaustive comparison, they do highlight the large difference at preserving multiple conserved quantities across a wide variety of examples.

For the implicit schemes, such as Implicit Midpoint method and MN-DMM schemes, we have used the improved Euler's method to obtain an initial guess for the fixed point iteration employed to solve the nonlinear or implicitly defined equations. For the choice of $\bm{f}^\tau$ for the MN-DMM in these tests, the improved Euler method was also chosen. For the sake of reproducibility, we have listed within each example their relevant problem parameters, time step size $\tau$, final time $T$, error tolerance for conserved quantities $\delta$, error tolerance for residual of the nonlinear equations $\epsilon$, and maximum number of fixed point iterations used per time step $K$. In all subsequent tables, we compare their maximum error in conserved quantities, as well as the mean fixed point iterations (FPIs) used for the implicit methods. Moreover, we compare the largest condition number of $\kappa(B)$ or $\kappa(A)$ encountered during simulation, for the Mixed MN-DMM and Mixed MN-DMM using SVD respectively. Note that for systems with a single conserved quantity, $\kappa(B)=1=\kappa(A)$, since $\Lambda^\tau {(\Lambda^\tau)}^{\top}$ is a scalar quantity as discussed in \Cref{sec: analytic m=1}.
\vskip -2mm
\subsection{Lotka-Volterra systems}
\label{sec: LV system}
As a first simple example, we illustrate MN-DMM for the two and three species Lotka-Volterra system, with one and two conserved quantities respectively. We first recall their definitions and conserved quantities.

In \cite[Example 5.2.1]{wan2017conservative}, analytic DMM schemes were derived for the two-species Lotka-Volterra system given by
\begin{align}
\bm F(\bm x, \dot{\bm x}) :=
\begin{pmatrix}
\dot{x} - x(a - b y) \\
\dot{y} - y(d x-c)
\end{pmatrix}, \label{eq:LV2}
\end{align}
for positive constants $a, b, c, d$. It is well-known that this system has a conserved quantity of the form
\begin{align}
\psi(\bm x) := a \log y-b y+c \log x-d x. \label{eq:LV2CQ}
\end{align}

Using $\tau = 0.1, T=10000, \delta = 1\times 10^{-15}, \epsilon = 1\times 10^{-15}, K=20$ and initial conditions $\bm x^0 = (0.3,0.7)^{\top}$ with $(a,b,c,d)=(1,2,3,4)$, we obtain the results showed in \Cref{tab:LV2}, which confirms the machine precision accuracy of the MN-DMM at preserving the conserved quantity $\psi$ of the two species Lotka-Volterra system. We also note that both the Implicit Midpoint method and MN-DMM methods utilized a similar number of fixed point iterations, with about $11\sim 12$ mean FPIs.

\begin{table}[h!]
\centering
\vskip -0.2cm
\begin{tabular}{|p{3.8cm}|c|c|c|}
\hline
\centering Numerical Method & $\scriptstyle\norm{\psi-\psi^0}_\infty$ & Mean FPIs & $\scriptstyle \norm{\kappa(\cdot)}_\infty$ \\
\hline
\hline
\centering RK4 & $1.279\times 10^{-1}$ & -- & -- \\
\hline
\centering Implicit Midpoint & $1.825\times 10^{-1}$ & 12.069 & --\\
\hline
\centering MN-DMM & $3.553\times 10^{-15}$ & 11.649 & --\\
\hline
\centering Mixed MN-DMM & $4.441\times 10^{-15}$ & 11.678 & 1.000\\
\hline
\centering Mixed MN-DMM (SVD) & $3.553\times 10^{-15}$ & 11.666 & 1.000\\
\hline
\end{tabular}
\caption{Two-species Lotka-Volterra system with\\ $\psi(x,y) = x-\log x+y-2 \log y$.}
\label{tab:LV2}
\end{table}
\begin{figure}[!ht]
\centering
\begin{subfigure}[b]{\textwidth}
  \centering
  \includegraphics[width=0.85\linewidth]{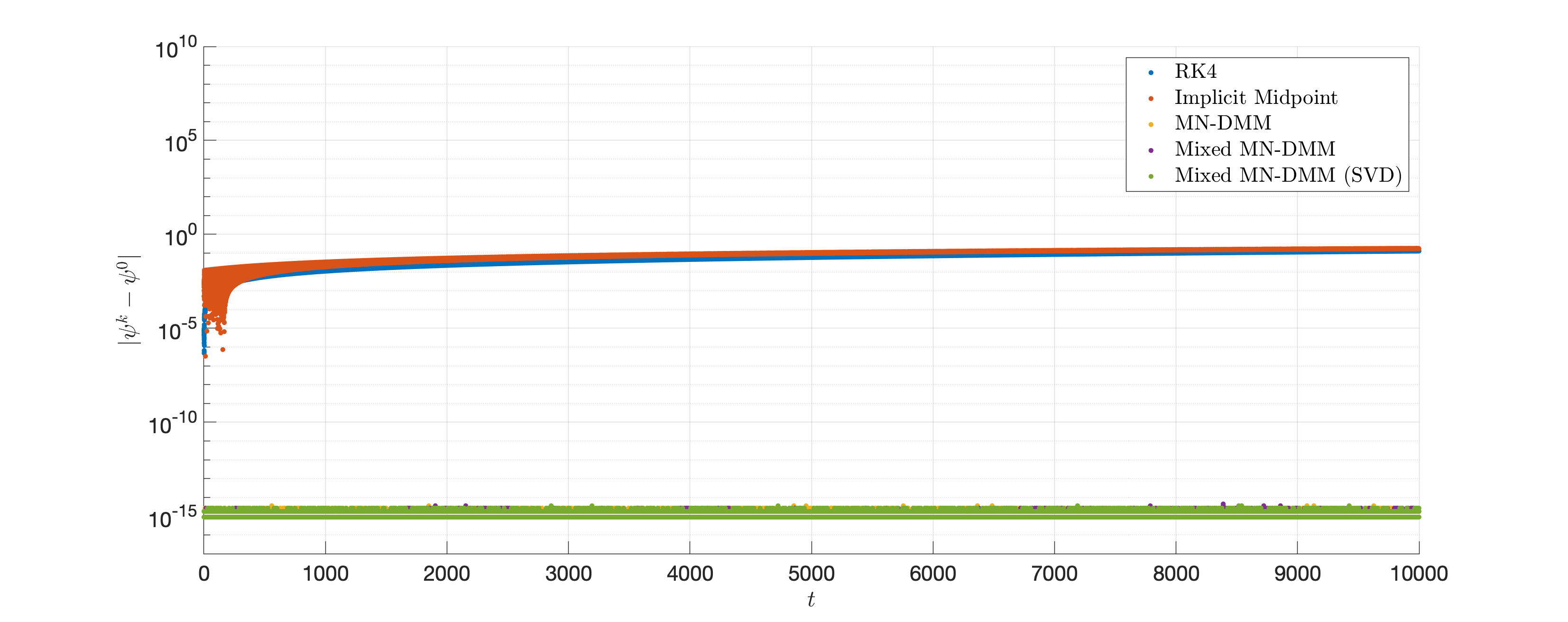}
\end{subfigure}
\hfil
\begin{subfigure}[b]{\textwidth}
  \centering
  \includegraphics[width=0.85\linewidth]{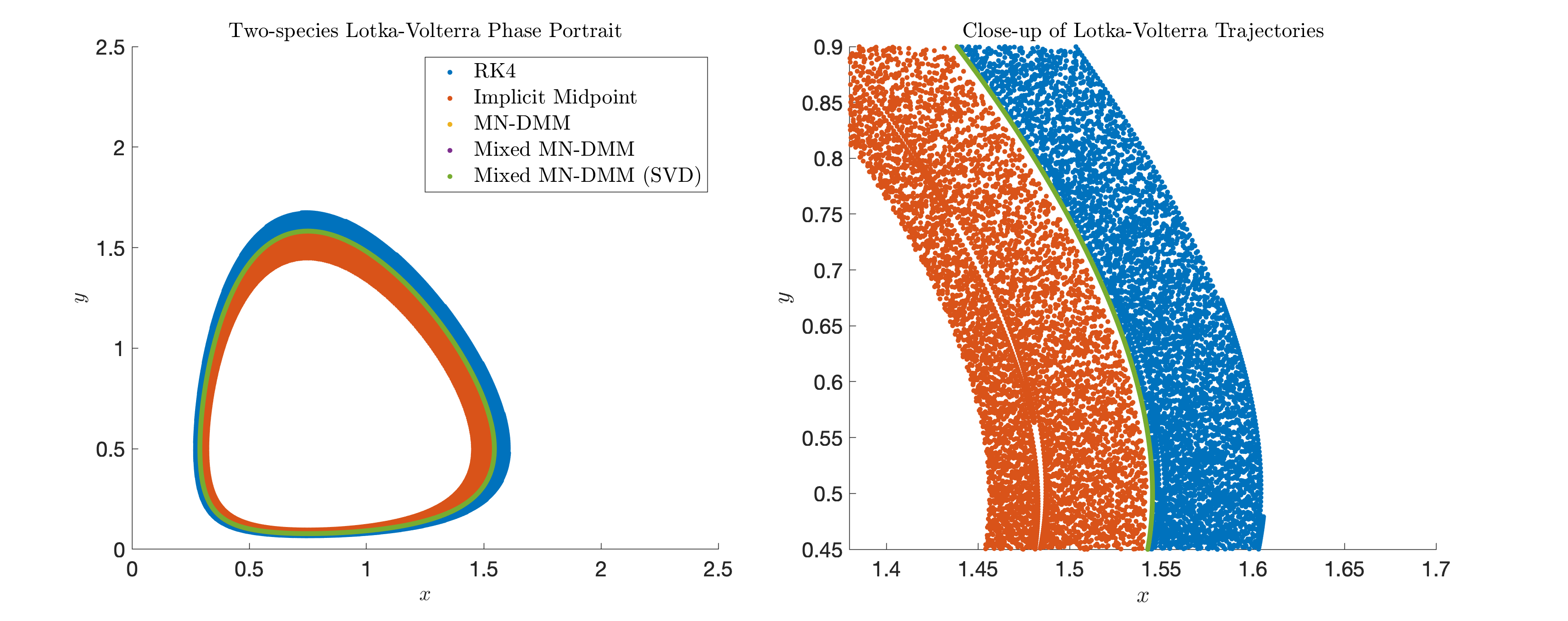}
\end{subfigure}
        \caption{Comparison of error in $\psi(\bm x)$ and trajectories for the two-species Lotka-Volterra problem.}
        \label{fig:LV2D}
\end{figure}
Moreover, \Cref{fig:LV2D} shows the trajectories of the Implicit Midpoint method and RK4 method drifting away from the level set of $\psi$. In contrast, the MN-DMM results show machine precision accuracy at remaining on the level set of $\psi$.

Extending to three-species, the Lotka-Volterra system takes the general form
\begin{align}
\bm F(\bm x, \dot{\bm x}) :=
\begin{pmatrix}
\dot{x} - x(a_{11}(x-\xi_1)+a_{12}(y-\xi_2)+a_{13}(z-\xi_3)) \\
\dot{y} - x(a_{21}(x-\xi_1)+a_{22}(y-\xi_2)+a_{23}(z-\xi_3))\\
\dot{z} - x(a_{31}(x-\xi_1)+a_{32}(y-\xi_2)+a_{33}(z-\xi_3))
\end{pmatrix}, \label{eq:LV3}
\end{align}
where $A=[a_{ij}]$ is a real-valued interaction matrix and $\bm \xi=(\xi_1,\xi_2,\xi_3)^{\top}$ is a fixed point of the system.
\cite{schi03a} showed that there are two conserved quantities 
\begin{align}
\bm \psi(\bm x) := 
\begin{pmatrix}
d_1(x-\xi_1\log x)+d_2(y-\xi_2\log y)+d_3(z-\xi_3\log z) \\
x^{\eta_1} y^{\eta_2} z^{\eta_3}
\end{pmatrix}, \label{eq:LV3CQ}
\end{align}
 if the diagonal matrix $D:=\text{diag}(d_1,d_2,d_3)$ and vector $\bm \eta:=(\eta_1,\eta_2,\eta_3)^{\top}$ satisfies
\begin{subequations}
\begin{align}
DA+A^{\top}D = 0, &\qquad \bm \eta^{\top} A = \bm 0.\label{eq:LV3-cond1}
\end{align}
\end{subequations}
In \cite[Example 5.2.2]{wan2017conservative}, analytic DMM schemes were derived for a special three-species system with a specific $A, D, \bm \xi, \bm \eta$.
Here we compare results using MN-DMM for the following example satisfying \eqref{eq:LV3-cond1},
$$
A= \begin{pmatrix}
0 &3 &-2\\
-3 &0 &1 \\
2 &-1 & 0
\end{pmatrix}, 
~~~\bm \xi = \begin{pmatrix}
1 \\
1 \\
1
\end{pmatrix}, 
~~~D = \text{diag}(1,1,1), ~~~\bm \eta = \begin{pmatrix}
1 \\
2 \\
3
\end{pmatrix}.
$$
Using $\tau = 0.05, T=30000, \delta = 1\times 10^{-15}, \epsilon = 1\times 10^{-15}, K=20$ and initial conditions $\bm x^0 = (0.2, 0.5, 0.3)^{\top}$, we obtain the result listed in \Cref{tab:LV3}.

\begin{table}[h!]
\centering
\vskip -0.2cm
\resizebox{\columnwidth}{!}{
\begin{tabular}{|p{3.8cm}|c|c|c|c|}
\hline
\centering Numerical Method & $\scriptstyle\norm{\psi_1-\psi_1^0}_\infty$ &
$\scriptstyle\norm{\psi_2-\psi_2^0}_\infty$ & Mean FPIs & $\scriptstyle \norm{\kappa(\cdot)}_\infty$ \\
\hline
\hline
\centering RK4 & $3.893\times 10^{-2}$ & $1.478\times 10^{-4}$ & -- & --\\
\hline
\centering Implicit Midpoint & $3.701\times 10^{0}$& $1.350\times 10^{-3}$ & 8.957 & -- \\
\hline
\centering MN-DMM & $3.553\times 10^{-15}$ & $1.003\times 10^{-15}$ & 12.205 & -- \\
\hline
\centering Mixed MN-DMM & $3.553\times 10^{-15}$ & $1.003\times 10^{-15}$ & 12.249 & $2.243\times 10^{6}$ \\
\hline
\centering Mixed MN-DMM (SVD)& $2.665\times 10^{-15}$ & $1.003\times 10^{-15}$ & 12.216 & $1.309 \times 10^{3}$\\
\hline
\end{tabular}
}
\caption{Three-species Lotka-Volterra system with \\$\bm \psi(\bm x) = \begin{pmatrix}x-\log x+y-2 \log y+z-3\log z \\ x y^2 z^3\end{pmatrix}.$}
\label{tab:LV3}
\end{table}
\begin{figure}[!ht]
\centering
\vskip -9mm
\begin{subfigure}[b]{\textwidth}
  \centering
  \includegraphics[width=0.9\linewidth]{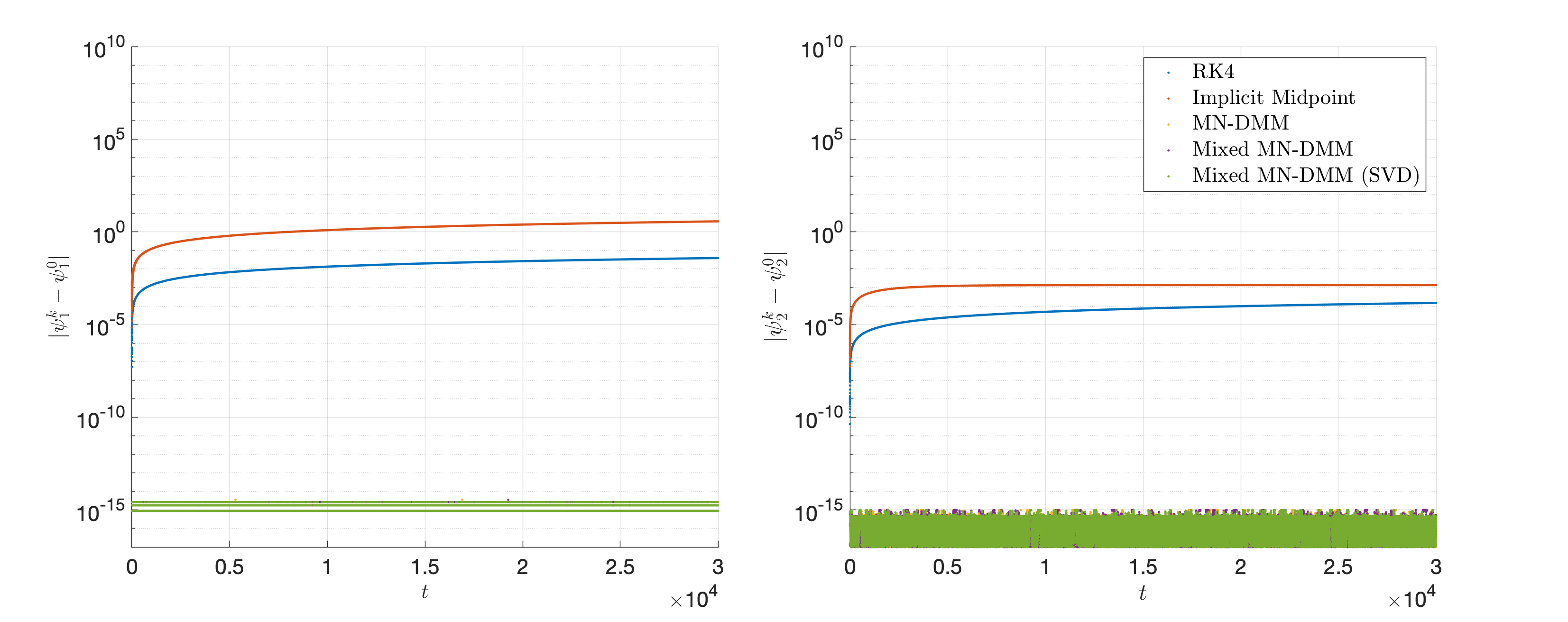}
\end{subfigure}
\hfil
\begin{subfigure}[b]{\textwidth}
  \centering
  \includegraphics[width=0.85\linewidth]{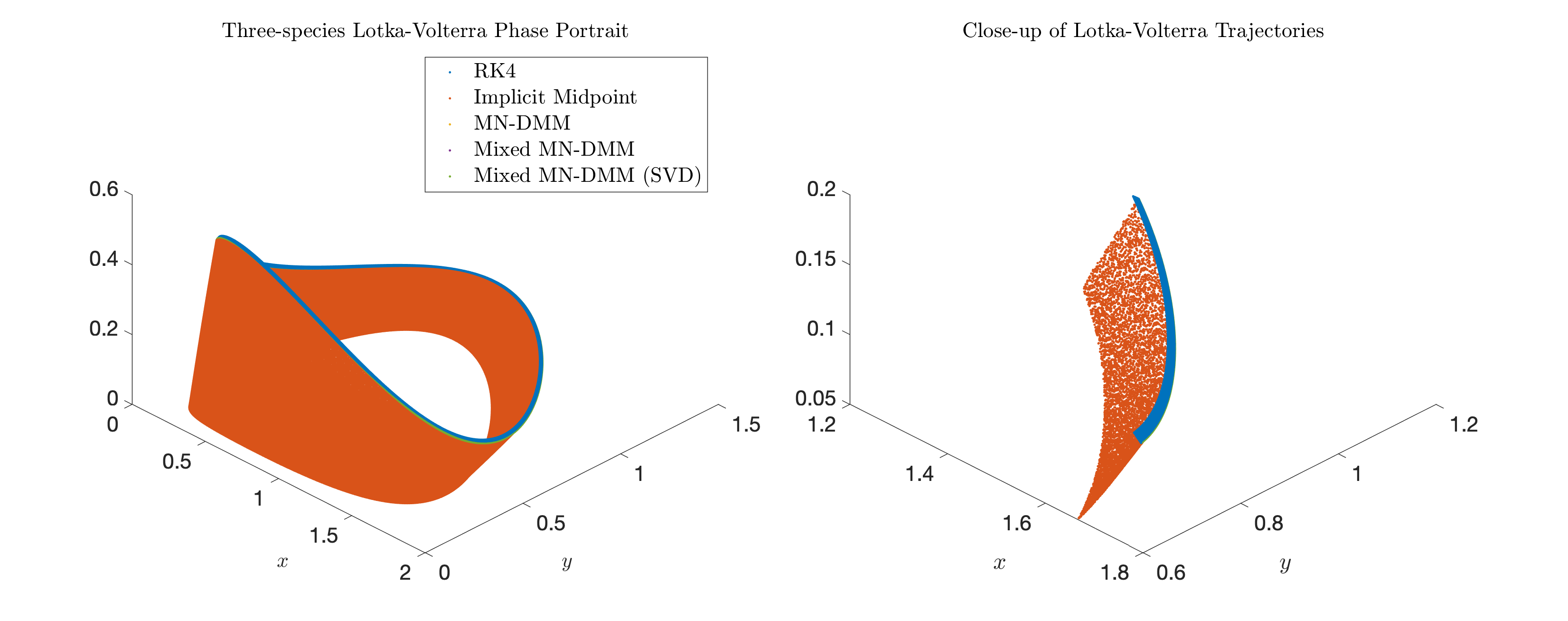}
\end{subfigure}
        \caption{Comparison of error in $\bm \psi(\bm x)$ and trajectories for the three-species Lotka-Volterra problem.}
        \label{fig:LV3D}
\end{figure}

Similar to the two-species case, \Cref{tab:LV3} shows machine precision accuracy at preserving the two conserved quantities $\bm \psi(\bm x)$ for the MN-DMM results. While MN-DMM did required a mean FPI of $\sim12$ over the Implicit Midpoint method 's mean FPI of $\sim9$, the MN-DMM results are the only methods not exhibiting large deviation of the level sets of $\bm \psi(\bm x)$ as shown in \Cref{fig:LV3D}, in contrast to the Implicit Midpoint and RK4 method. Moreover, as this example involves more than one conserved quantity, \Cref{tab:LV3} now shows a smaller condition number for the associated linear system when SVD is used in the Mixed MN-DMM approach.
\vskip -20mm
\subsection{Planar restricted three-body problem}
\label{sec: res3body}
$ $\\
In \cite[Example 5.3]{wan2017conservative}, the planar restricted 3-body problem involving the Arenstorf orbit parameters was considered and an analytic DMM scheme was derived, albeit with much effort using divided difference calculus. Here we consider the same example but with much less effort to derive the conservative scheme using the MN-DMM approach.

For completeness, we first briefly recall the planar restricted three-body problem, which describes the gravitational motion of three bodies in a plane with a negligible mass, such as the Earth--Moon--Satellite system. The equations of motions are
\begin{equation}
\boldsymbol F(\bm x,\dot{\bm x}) 
:=
\begin{pmatrix}\dot{x_1} - y_1 \\
\dot{x_2} -y_2\\
\dot{y_1} -\left(x_1+2y_2-\dfrac{\alpha(x_1-\beta)}{((x_1-\beta)^2+x_2^2)^{\frac{3}{2}}}-\dfrac{\beta(x_1+\alpha)}{((x_1+\alpha)^2+x_2^2)^{\frac{3}{2}}}\right)\\
\dot{y_2} -\left(x_2-2y_1-\dfrac{\alpha x_2}{((x_1-\beta)^2+x_2^2)^{\frac{3}{2}}}-\dfrac{\beta x_2}{((x_1+\alpha)^2+x_2^2)^{\frac{3}{2}}}\right)\end{pmatrix}, \label{3bodySys}
\end{equation}
where $\bm x = (x_1,x_2,y_1,y_2)$ are the relative positions and momenta of the satellite to the center of mass between the Earth and Moon, with $\alpha, \beta$ being relative masses of the two bodies satisfying $\alpha+\beta=1$. It is well-known that \eqref{3bodySys} has a conserved quantity called the Jacobi integral $J$ given by,
\[
 J(\boldsymbol x)=\dfrac{x_1^2 + x_2^2- y_1^2-y_2^2}{2} + \dfrac{\alpha}{((x_1-\beta)^2+x_2^2)^{\frac{1}{2}}}+\dfrac{\beta}{((x_1+\alpha)^2+x_2^2)^{\frac{1}{2}}}.
\]
We consider the Arenstorf orbit period $P=17.0652165601579625588917206249$ and parameter $\alpha=0.012277471$ were used with initial conditions,
$$\bm x^0 = (0.994, 0, 0, -2.00158510637908252240537862224)^{\top}.$$ Using the solver parameters $T=P\times 1.015, \tau = T\times 10^{-6}, \delta = 1\times 10^{-15}, \epsilon = 1\times 10^{-15}, K=20$, we obtained the error in the Jacobi integral in \Cref{tab:res3body}.
\begin{table}[h!]
\centering
\vskip -0.2cm
\begin{tabular}{|p{3.8cm}|c|c|c|}
\hline
\centering Numerical Method & $\scriptstyle\norm{J-J^0}_\infty$ & Mean FPIs & $\scriptstyle \norm{\kappa(\cdot)}_\infty$ \\
\hline
\hline
\centering RK4 & $5.793\times 10^{-8}$ & -- & --\\
\hline
\centering Implicit Midpoint & $1.921\times 10^{2}$ & 2.468 & -- \\
\hline
\centering MN-DMM & $6.639\times 10^{-14}$ & 17.310 & --\\
\hline
\centering Mixed MN-DMM & $6.639\times 10^{-14}$ & 17.310 & 1.000\\
\hline
\centering Mixed MN-DMM (SVD)& $6.639\times 10^{-14}$ & 17.310 & 1.000\\
\hline
\end{tabular}
\caption{Planar restricted three-body problem with \\conserved quantity $J(\bm x)$}
\label{tab:res3body}
\end{table}

\begin{figure}[!ht]
\centering
\begin{subfigure}[b]{\textwidth}
  \centering
  \includegraphics[width=0.95\linewidth]{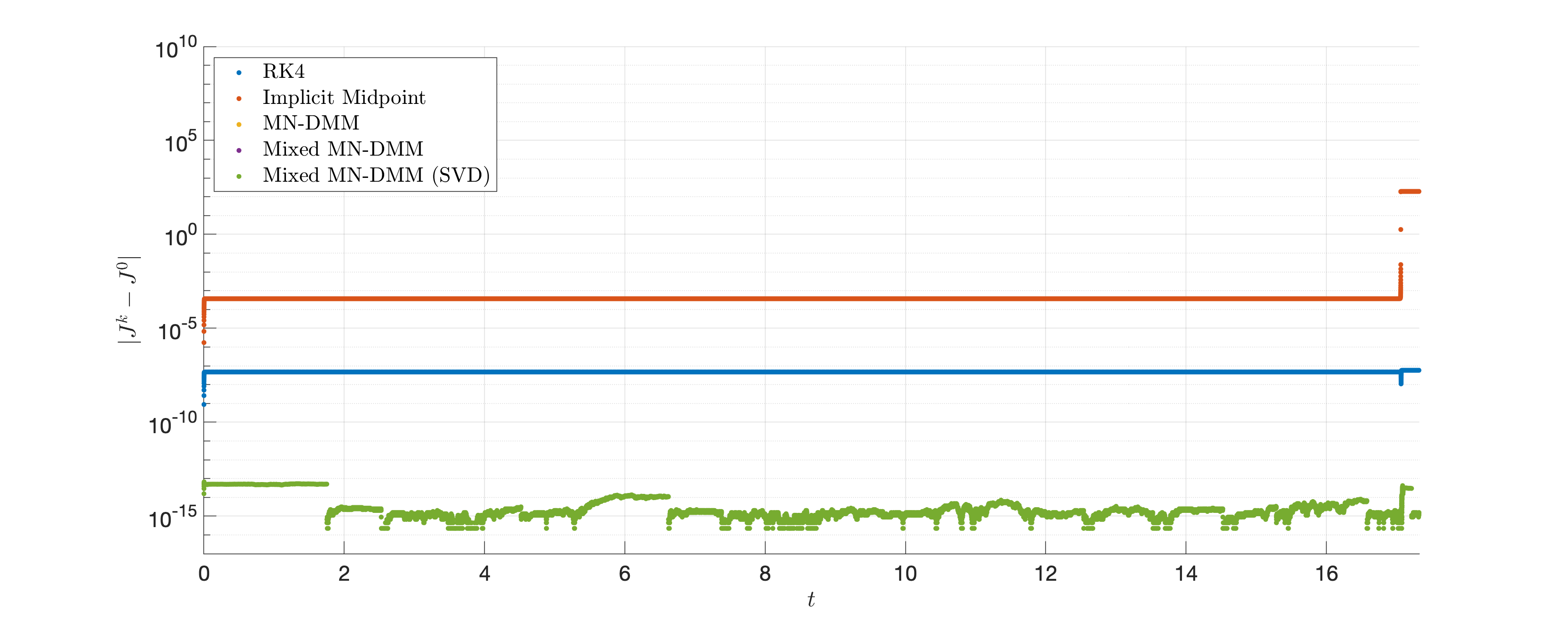}
\end{subfigure}
\hfil
\begin{subfigure}[b]{\textwidth}
  \centering
  \includegraphics[width=0.85\linewidth]{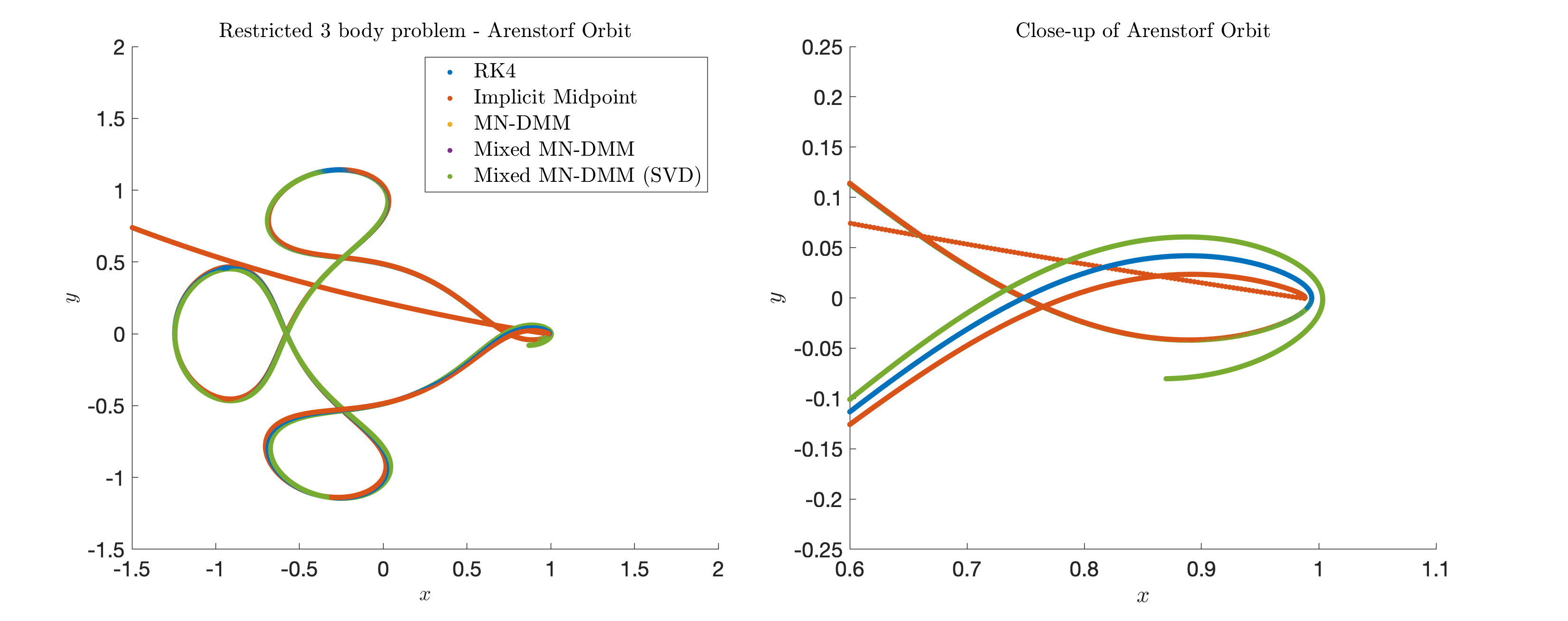}
\end{subfigure}
        \caption{Comparison of error in Jacobi integral and trajectories for the Arenstorf orbit.}
        \label{fig:arenstorf}
\end{figure}
As \Cref{fig:arenstorf} illustrates, all methods were able to reproduce the Arenstorf orbit qualitatively over one period $P$. However, shortly after one period, the Implicit Midpoint method results in a nonphysical trajectory, with several orders of magnitude jump in the error of the Jacobi integral due to the nonconvergence of its fixed point iterations. While the results from the MN-DMM approach do not show an exact periodic orbit, their trajectories beyond one period are close to that of the RK4 method, which is expected due to its higher order accuracy than the presented MN-DMM methods.

\subsection{Lorenz system}
\label{sec: Lorenz system}
In \cite{wan2017conservative}, analytic DMM scheme was also derived for time-dependent conserved quantities for dissipative systems, such as the damped harmonic oscillator. As another interesting example with time-dependent conserved quantities, we consider the Lorenz system for $\bm x = (x,y,z)$, 
\begin{align*}
\bm F(\bm x, \dot{\bm x}) :=
\begin{pmatrix}
\dot{x} - \sigma(y-x) \\
\dot{y} - x(\rho-z)-y\\
\dot{z} - xy-\beta z
\end{pmatrix}
\end{align*} which has six conserved quantities over different sets of positive parameters $\sigma, \rho, \beta$ in nonchaotic regime \cite{AblowitzSegur81,kus83}.
Specifically, for the parameters $\sigma=1/3, \rho=400$ and $\beta=0$, \cite{kus83} showed that there exists a conserved quantity of the form,
$$
\psi(t,\bm x)= \left(x^4-\frac{4}{3}x^2z-\frac{4}{9}y^2-\frac{8}{9}xy+\frac{1600}{3}x^2\right)e^{4t/3}.
$$

Using $\tau = 0.001, T=5, \delta = 1\times 10^{-15}, \epsilon = 1\times 10^{-15}, K=20$ and initial conditions $\bm x^0 = (0.1, 0, 0)^{\top}$, we obtain the error in $\psi(t,\bm x)$. 

\begin{table}[h!]
\centering
\begin{tabular}{|p{3.8cm}|c|c|c|}
\hline
\centering Numerical Method & $\scriptstyle\norm{\psi-\psi^0}_\infty$ & Mean FPIs & $\scriptstyle \norm{\kappa(\cdot)}_\infty$ \\
\hline
\hline
\centering RK4 & $2.916\times 10^{-3}$ & -- & -- \\
\hline
\centering Implicit Midpoint & $7.971\times 10^{1}$ & 18.601 & -- \\
\hline
\centering MN-DMM & $4.425\times 10^{-8}$ & 19.990 & -- \\
\hline
\centering Mixed MN-DMM & $4.425\times 10^{-8}$ & 19.990 & 1.000\\
\hline
\centering Mixed MN-DMM (SVD)& $4.425\times 10^{-8}$ & 19.990 & 1.000\\
\hline
\end{tabular}
\caption{Lorenz system with time-dependent conserved quantity \\ $\psi(t,\bm x) =  \left(x^4-\frac{4}{3}x^2z-\frac{4}{9}y^2-\frac{8}{9}xy+\frac{1600}{3}x^2\right)e^{4t/3}$.}
\label{tab:lorenz}
\end{table}
\begin{figure}[!ht]
\centering
\vskip -4mm
\begin{subfigure}[b]{\textwidth}
  \centering
  \includegraphics[width=0.9\linewidth]{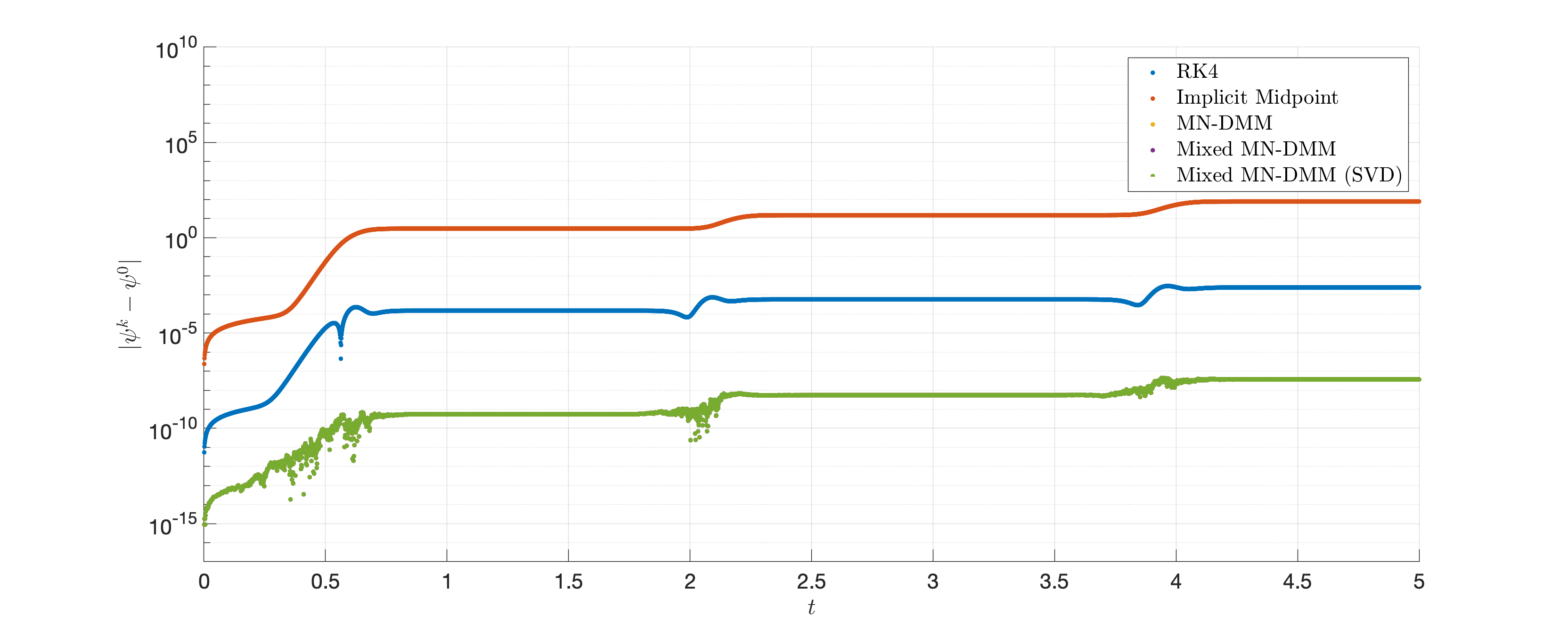}
\end{subfigure}
\hfil
\begin{subfigure}[b]{0.85\textwidth}
  \centering
  \includegraphics[width=\linewidth]{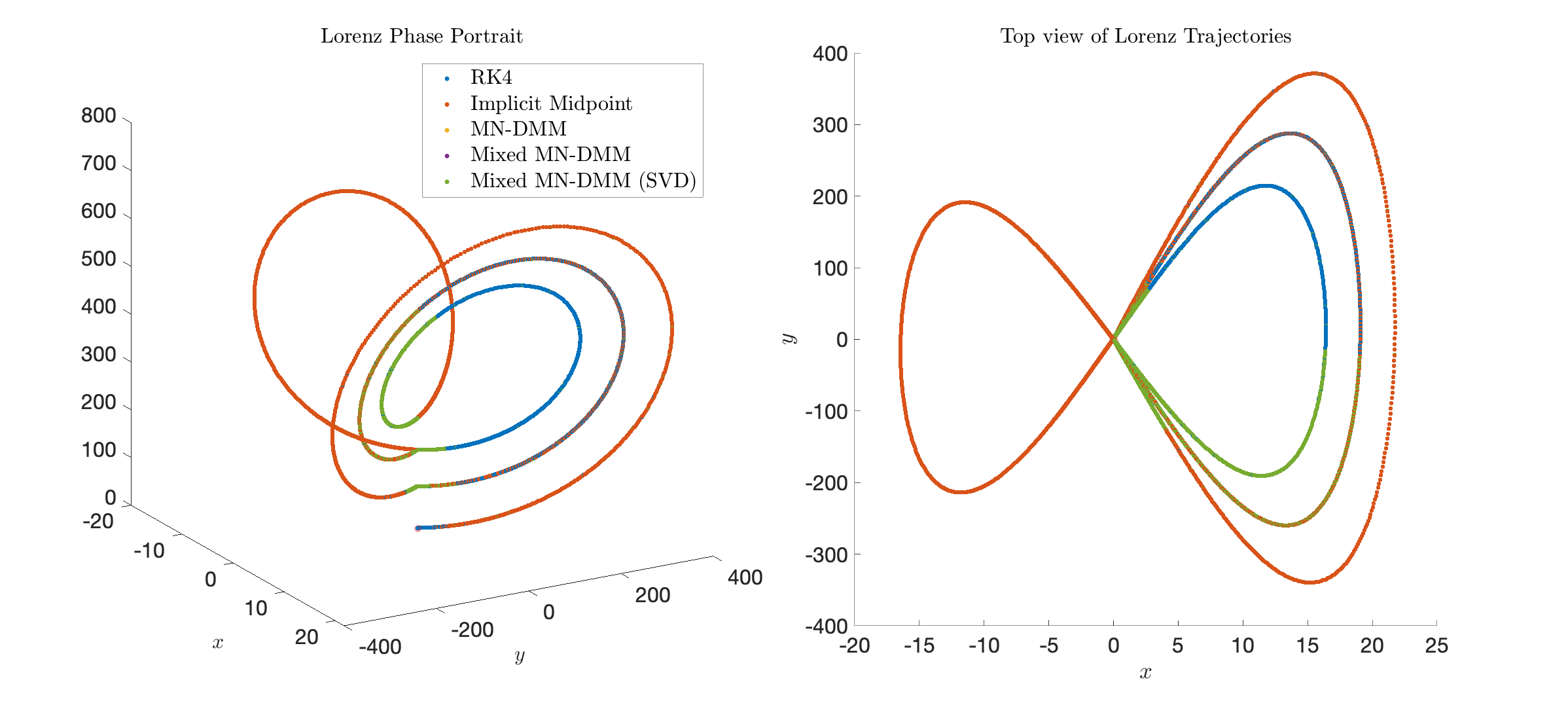}
\end{subfigure}
        \caption{Comparison of error in $\psi(t,\bm x)$ and trajectories for the Lorenz system.}
        \label{fig:lorenz}
        \vskip -4mm
\end{figure}
\Cref{tab:lorenz} indicates that machine precision accuracy for the time-dependent conserved quantity was not obtained using the MN-DMM approach. This is due to the stiffness of the problem as indicated by the high average number of FPIs for both the Implicit Midpoint method and the MN-DMM approach. Moreover, it can be observed in \Cref{fig:lorenz} that fast transient dynamics occurs when the solution loops back toward the origin on the $xy$--plane, corresponding to the three apparent ``jumps" in the error of $\psi$. Nevertheless, the Implicit Midpoint method has the largest error of $\sim 10^1$ in the conserved quantity, leading to an incorrect transient part of its trajectory located in the $x<0$ region, as depicted in  \Cref{fig:lorenz}. In contrast, the RK4 method and the MN-DMM approach have respective errors of $\sim 10^{-3}$ and $\sim 10^{-8}$ in the time-dependent conserved quantity $\psi$, with their trajectories remaining in the $x>0$ region.
\vskip -4mm
\subsection{$N$-point vortex problem on the unit sphere}
\label{sec: point vortex}
In \cite[Example 4.5]{manybody2022}, analytic DMM scheme was derived for the classical $N$-point vortex problem on the unit sphere, which is an idealized model of approximating the solution to the incompressible Euler's equation on the unit sphere, given by
\begin{align}
\bm F(\bm x, \dot{\bm x}) := 
\dot{\bm x}_i-\dfrac{1}{4\pi}\sum\limits_{j=1, j\neq i}^N \Gamma_j \dfrac{{\bm x}_j\times {\bm x}_i}{1-{\bm x}_i\cdot {\bm x_j}}
 =\boldsymbol 0, \label{eq:PV}
\end{align} where $\bm x=({\bm x}_1,\dots, {\bm x}_n)^{\top}$ with ${\bm x}_i\in \mathbb{S}^2$ being the position of the $i$-th point vortex on the unit sphere and $\Gamma_i$ being the vortex strength of the $i$-th vortex. The point vortex equations on the unit sphere~\eqref{eq:PV} possess four conserved quantities, given by the momentum vector $\bm P\in \mathbb{R}^3$ and the Hamiltonian $H$, which are
\begin{align}\label{eq:CQPV}
\begin{split}
\bm P(\bm x) &:= \sum\limits_{i=1}^N \Gamma_i {\bm x}_i,
\end{split}
\begin{split}
H(\bm x) &:= -\dfrac{1}{4\pi}\sum\limits_{1\leq i< j\leq N} \Gamma_i \Gamma_j \log  (1-{\bm x}_i\cdot {\bm x}_j).
\end{split}
\end{align}
An analytic DMM scheme was derived in \cite{manybody2022} with significant computation effort to verify the discrete multiplier conditions, in contrast to the MN-DMM approach. Using $N=100$ randomly generated vortices and the solver parameters $\tau = 0.1, T=200, \delta = 1\times10^{-15}, \epsilon = 1\times10^{-15}$ and $K=20$, we obtain the error in four conserved quantities given in \Cref{tab:PVsphere}.

\begin{table}[h!]
\centering
\vskip -0.2cm
\resizebox{\columnwidth}{!}{
\begin{tabular}{|p{3.8cm}|c|c|c|c|}
\hline
\centering Numerical Method & $\scriptstyle\norm{\bm{P}-\bm{P}^0}_\infty$ & $\scriptstyle\norm{H-H^0}_\infty$ & Mean FPIs & $\scriptstyle \norm{\kappa(\cdot)}_\infty$\\
\hline
\hline
\centering RK4 & $3.022\times 10^{-16}$ & $1.360\times 10^{-6}$ & -- & -- \\
\hline
\centering Implicit Midpoint & $3.193\times 10^{-16}$ & $1.240\times 10^{-7}$ & 20.000 & --\\
\hline
\centering MN-DMM & $2.705\times 10^{-16}$ & $1.025\times 10^{-15}$ & 4.670 & --\\
\hline
\centering Mixed MN-DMM & $3.243\times 10^{-16}$ & $1.022\times 10^{-15}$ & 4.652 & 11.58\\
\hline
\centering Mixed MN-DMM (SVD)& $3.243\times 10^{-16}$ & $1.022\times 10^{-15}$ & 4.652 & 3.403\\
\hline
\end{tabular}
}
\caption{Point vortices on the unit sphere with conserved quantities $\bm P(\bm x)$ and $H(\bm x)$.}
\label{tab:PVsphere}
\end{table}
\vskip -5mm
\begin{figure}[!ht]
\centering
\begin{subfigure}[b]{\textwidth}
  \centering
  \includegraphics[width=0.85\linewidth]{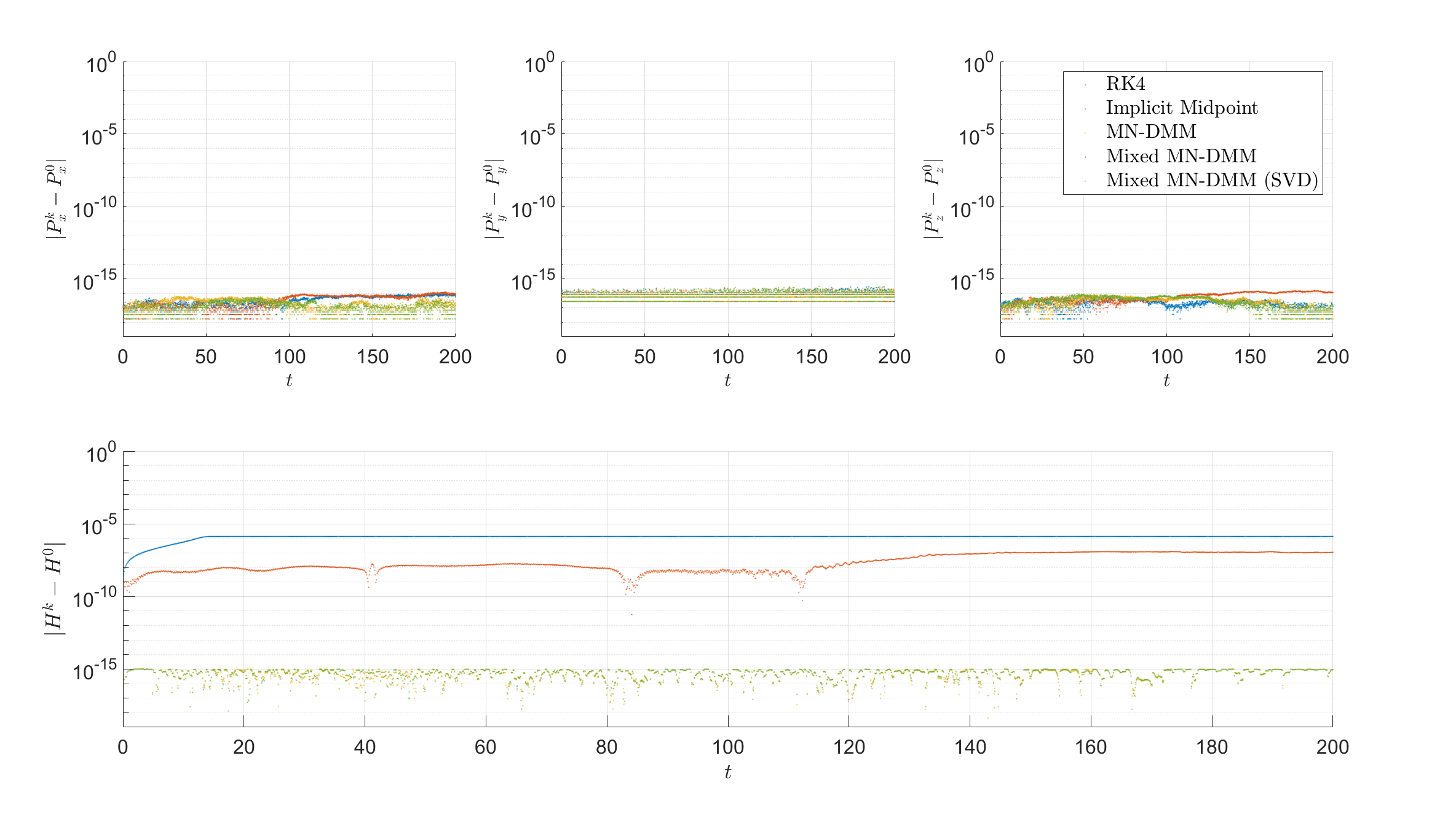}
\end{subfigure}
\hfil
\begin{subfigure}[b]{\textwidth}
  \centering
  \includegraphics[width=0.85\linewidth]{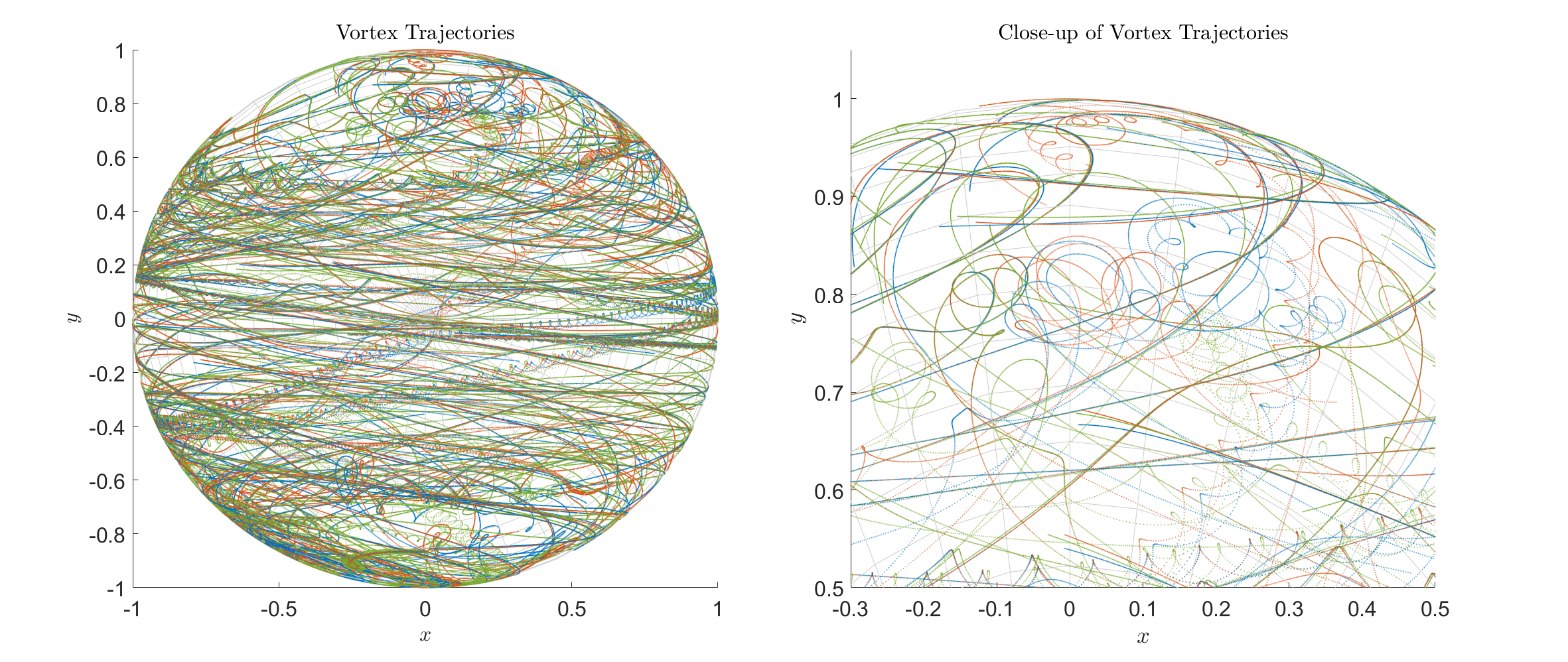}
\end{subfigure}
\caption{Comparison of error in conserved quantities and trajectories for the point vortex problem.}
\label{fig:PVSphere}
\end{figure}

\Cref{tab:PVsphere} indicates that this problem is relatively well-conditioned, with the MN-DMM approach converging faster than the Implicit Midpoint method. Moreover, as \Cref{fig:PVSphere} illustrates, all methods can preserves the momentum vector $\bm P$ up to machine precision. This is expected since the conserved quantities $\bm P$ are linear invariants, see \cite{hair06Ay}. On the other hand, only the MN-DMM approach is able to preserve the Hamiltonian $H$. While both the RK4 and Implicit Midpoint methods have error in Hamiltonian of $10^{-7}\sim 10^{-6}$, the observed trajectories are in stark contrast to the MN-DMM ones on a relatively short integration time of $T=200$. This is consistent with the observations made in \cite[Example 4.5]{manybody2022} using the analytic DMM scheme for this problem. Thus, for larger number of vortices and longer term integration, large deviation in trajectories are likely to occur when the error in Hamiltonian is not close to machine precision.

\subsection{Geodesic curve on Schwarzschild Geometry}
\label{sec: geodesic}

For the final example, we apply the MN-DMM approach to solve for geodesic curves on an $n$-dimensional pseudo-Riemannian manifold.  Specifically, we study geodesics for the Schwarzschild metric via the evolution of test particles in a spherically symmetric gravitational field. We refer to \cite{carroll_2019,Godinho_2014,o1983semi} for details on the following system. Recall that geodesic curves locally satisfy the first order system of ordinary differential equations
\begin{equation}
\bm F(\bm x,\bm y,\dot{\bm x},\dot{\bm y}) := \begin{pmatrix}
\left[\dot{x}^l - y^l\right]_{1\leq l\leq n}\\
\left[\dot{y}^l +  \sum\limits_{j,k=1}^n\Gamma^l_{j,k}(\bm x) y^jy^k\right]_{1\leq l\leq n}
\end{pmatrix} = \bm 0, \label{eq:gdSys}
\end{equation}
where $\Gamma^i_{j,k}$ are Christoffel symbols of the second kind, cf. \cite[Chap.3]{carroll_2019} or \cite[Chap. 3]{o1983semi}.  A well-known conserved quantity is the speed  \cite[Chap. 5.4]{carroll_2019} given by
\begin{align}
S(\bm x,& \bm y) =  \sum_{i,j=1}^n g_{ij}(\bm x)y^iy^j \label{eq:gdSpeed}
\end{align}
where $g_{ij}(\bm x)$ denotes the Riemannian metric tensor 
\cite[Sec. 3.8 and Appendix B]{carroll_2019}. As a concrete example, we consider the Schwarzschild metric, which is a radially symmetric solution to Einstein's equation in vacuum. In Schwarzschild coordinates $\bm x = (t,r,\theta,\phi)$ and $\bm y = (t',r',\theta',\phi')$, it is represented by the diagonal matrix
\begin{equation}
g(\bm x) = \text{diag}\left(1-\dfrac{r_s}{r},  -\left(1-\dfrac{r_s}{r}\right)^{-1}, -r^2, -r^2\sin^2\theta \right), \label{eq:SMetric}
\end{equation} 
where $r_s = \frac{2GM}{c^2}$ is the Schwarzschild radius.
In this setting, there are five conserved quantities. Indeed, using the spherical symmetries of this metric, it can be shown 
that the energy $E$ and angular momentum $\bm L$ are conserved:
\begin{align*}
E(\bm x, \bm y) &= \left(1-\frac{r_s}{r}\right)t',\\
\bm L(\bm x, \bm y) &= \begin{pmatrix} r^2\sin(\theta)\phi' \\ 
                         r^2(\cos(\phi)\theta'-\cos(\theta)\sin(\phi)\phi') \\
                         r^2(\sin(\phi)\theta'-\cos(\theta)\cos(\phi)\phi')
         \end{pmatrix}.
\end{align*}
Moreover, the expression in \eqref{eq:gdSpeed} reduces to
\begin{align*}
S(\bm x, \bm y) &= \left(1-\frac{r_s}{r}\right)t'^2-\left(1-\dfrac{r_s}{r}\right)^{-1}r'^2-r^2\theta'^2-r^2\sin^2 \theta \phi'^2.
\end{align*}

Due to the complexity of these expressions, significant computation effort would be required to derive an analytic DMM scheme for \eqref{eq:gdSys} to preserve these five conserved quantities. In contrast, the MN-DMM approach requires relatively minimal effort to implement. We compare their numerical results using the solver parameters $\tau = 1/3, T=200, \delta = 1\times10^{-15}, \epsilon = 1\times10^{-15}, K=20$. We have set $G, M, c$ to unity for simplicity, and used the initial conditions
\begin{align*}
    \bm x^0 &= \begin{pmatrix} 0, 37.338379348829989, \pi/2, 3.006861595479139 \end{pmatrix}^\top,\\
    \bm y^0 &= \begin{pmatrix} 1, -0.990937492340824, 0, 0.003597472991852\end{pmatrix}^\top.
\end{align*}

As \Cref{tab:Schwarzschild} shows, the MN-DMM schemes are the only ones able to preserve all five conserved quantities up to machine precision. In contrast, the RK4 method was unstable at $\tau=1/3$ and the Implicit Midpoint method had errors in the conserved quantities between $10^{-4}\sim10^{-3}$. Due to the intricate short time dynamics of passing near the Schwarzschild radius, both the Implicit Midpoint method and MN-DMM required a similar number of fixed point iterations of $18\sim19$, with the maximum of 20. Also, the condition number for the Mixed MN-DMM using SVD approach is nearly seven orders of magnitude smaller than the Mixed MN-DMM.

\begin{table}[h!]
\centering 
\vskip -0.2cm
\resizebox{\columnwidth}{!}{
\begin{tabular}{|c|c|c|c|c|c|}
\hline
Numerical Method & $\scriptstyle\norm{S-S^0}_\infty$ & $\scriptstyle\norm{E-E^0}_\infty$ & $\scriptstyle\norm{\bm L-\bm L^0}_\infty$ & Mean FPIs & $\scriptstyle \norm{\kappa(\cdot)}_\infty$\\
\hline
\hline
RK4 & NaN & NaN & NaN & -- & --  \\
\hline
Implicit Midpoint & $2.590\times 10^{-4}$ & $3.624\times 10^{-4}$ & $4.590\times 10^{-3}$ & 18.863 & --\\
\hline
MN-DMM & $7.896\times 10^{-15}$ & $1.221\times 10^{-15}$ & $1.579\times 10^{-14}$ & 19.142 & -- \\
\hline
Mixed MN-DMM & $4.816\times 10^{-15}$ & $9.992\times 10^{-16}$ & $8.464\times 10^{-15}$ & 19.273 & $1.023\times 10^{13}$\\
\hline
{Mixed MN-DMM (SVD)} & $9.867\times 10^{-15}$ & $1.332\times 10^{-15}$ & $1.921\times 10^{-14}$ & 19.347 & $5.062\times 10^{5}$\\
\hline
\end{tabular}
}
\caption{Geodesic curves on Schwarzschild Geometry with conserved quantities $S, E, \bm L$ ($\tau=1/3$)}
\label{tab:Schwarzschild}
\end{table}
\begin{figure}[!ht]
\centering
\vskip -4mm
\begin{subfigure}[b]{\textwidth}
  \centering
  \includegraphics[width=\linewidth]{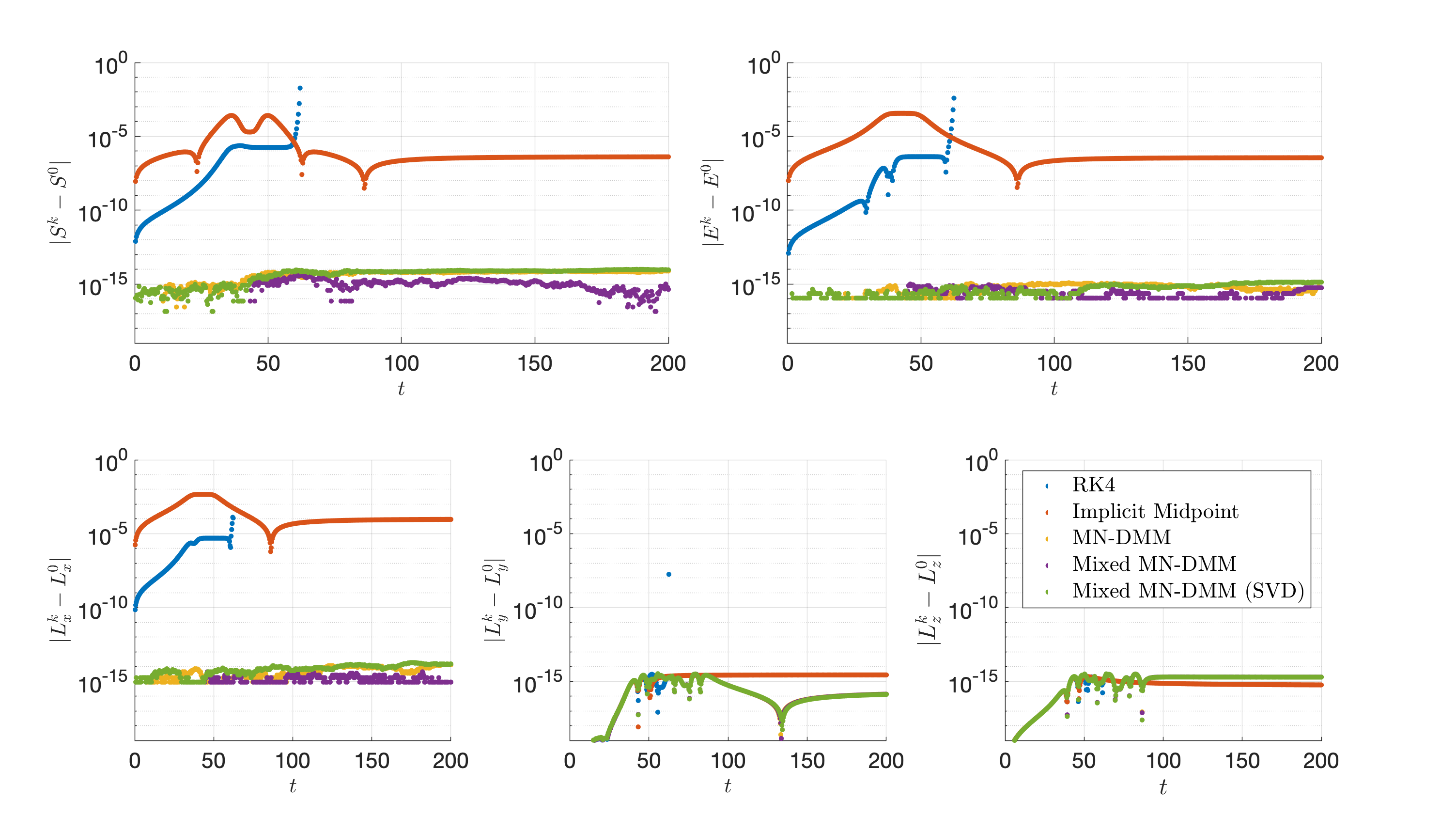}
\end{subfigure}
\hfil
\begin{subfigure}[b]{\textwidth}
  \centering
  \includegraphics[width=\linewidth]{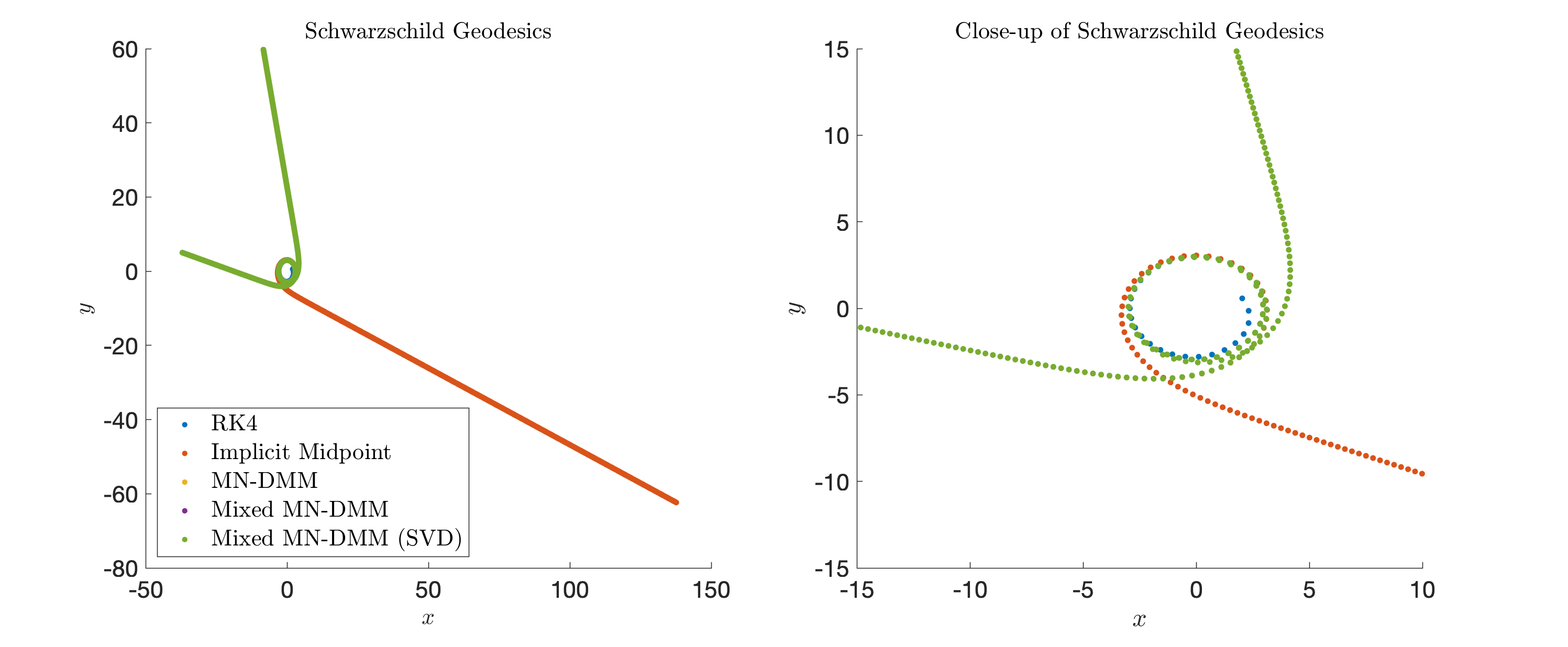}
\end{subfigure}
        \caption{Comparison of error in conserved quantities and geodesics using $\tau=1/3$.}
        \label{fig:schwarzschild-0}
\end{figure}

\begin{figure}[!ht]
\centering
\begin{subfigure}[b]{\textwidth}
  \centering
  \includegraphics[width=0.81\linewidth]{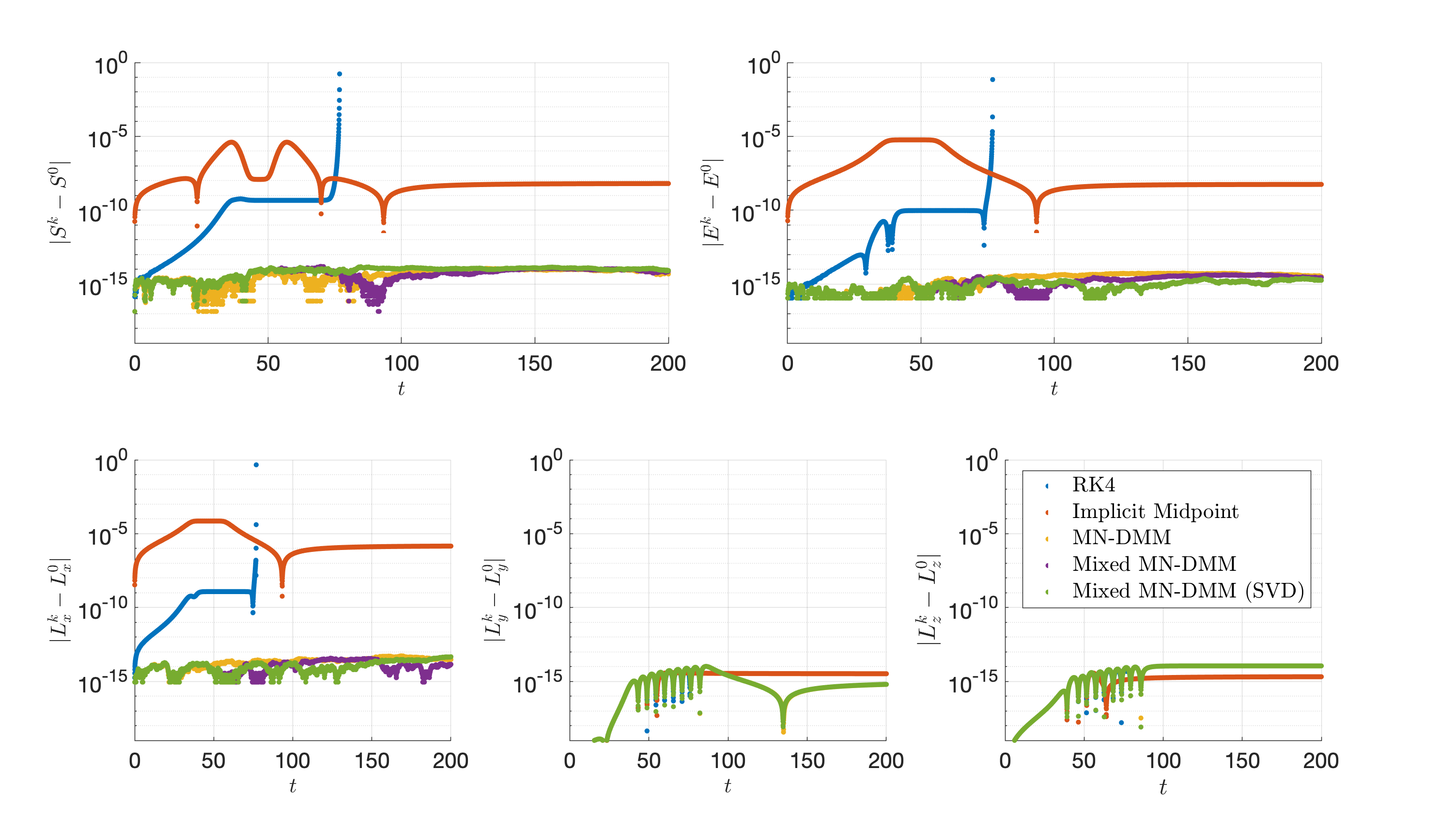}
  \vskip -4mm
  \caption{$\tau=1/3\times 2^{-3}$}
\end{subfigure}
\hfil
\begin{subfigure}[b]{\textwidth}
  \centering
  \includegraphics[width=0.81\linewidth]{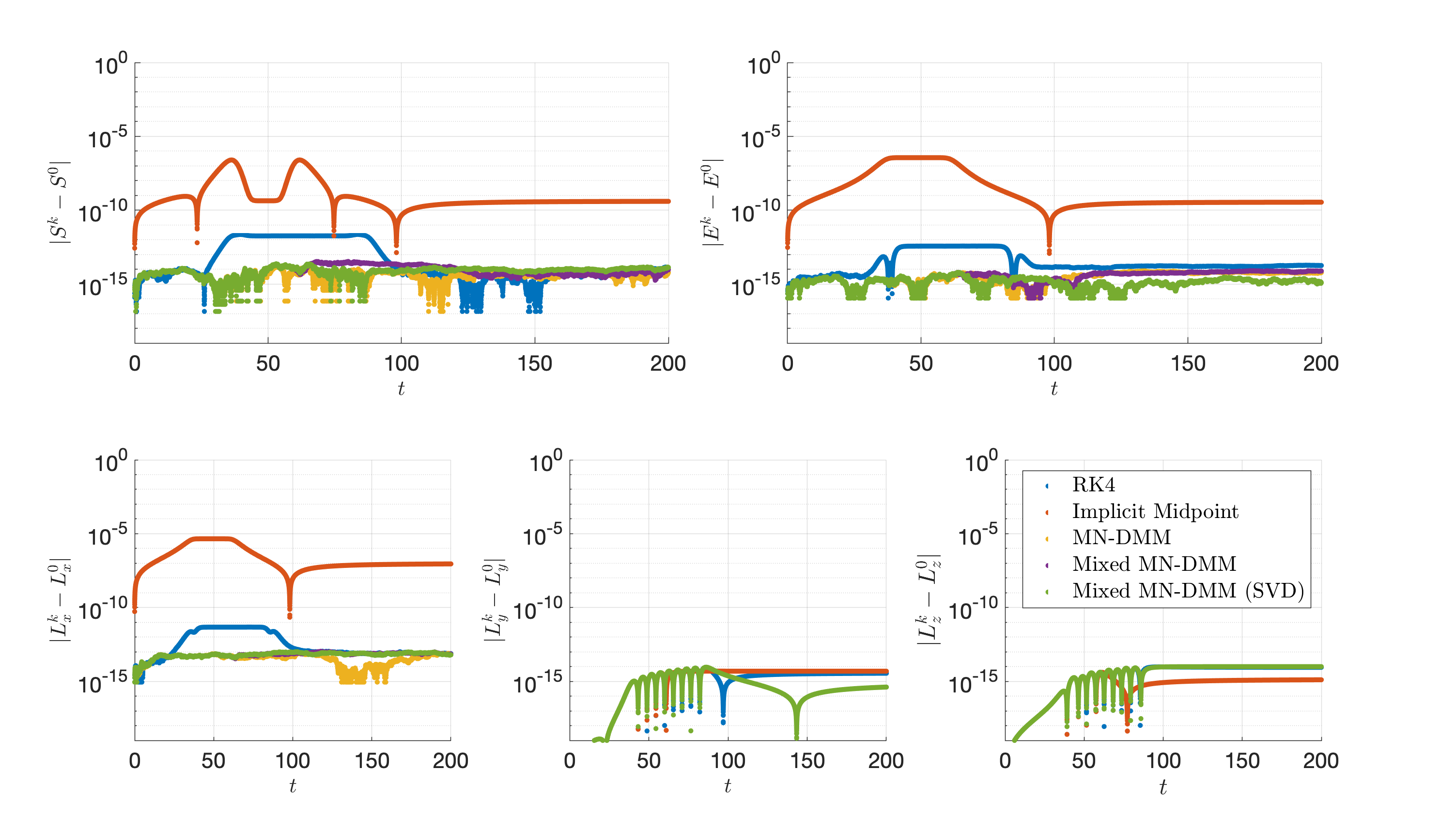}
  \vskip -4mm
  \caption{$\tau=1/3\times 2^{-5}$}
\end{subfigure}
\hfil
\begin{subfigure}[b]{\textwidth}
  \centering
  \includegraphics[width=0.81\linewidth]{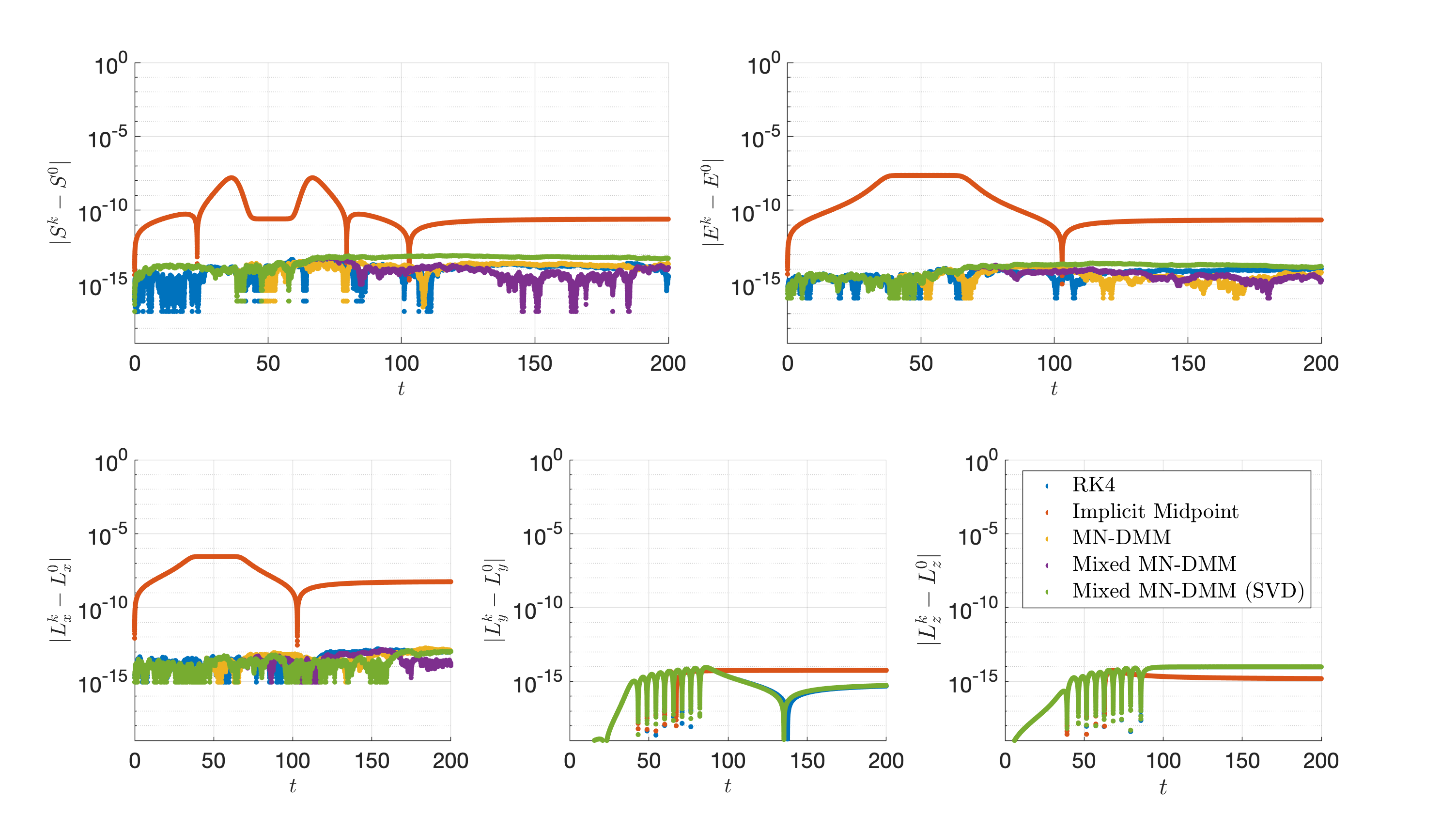}
  \vskip -4mm
  \caption{$\tau=1/3\times 2^{-7}$}
\end{subfigure}
        \caption{Comparison of error in conserved quantities for various $\tau$.}
        \label{fig:schwarzschildCQ-1}
\end{figure}

\begin{figure}[!ht]
\centering
\begin{subfigure}[b]{\textwidth}
  \centering
  \includegraphics[width=0.81\linewidth]{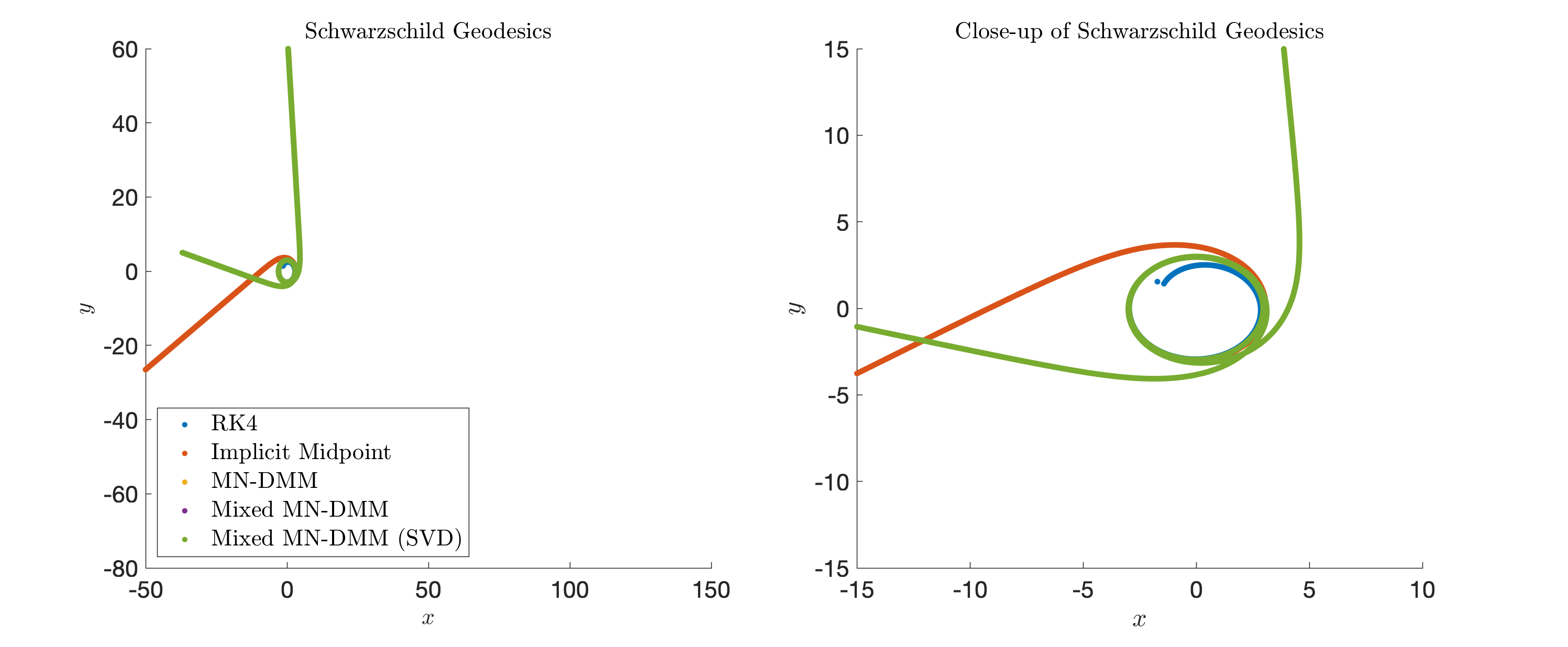}
  \vskip -2mm
  \caption{$\tau=1/3\times 2^{-3}$}
\end{subfigure}
\hfil
\begin{subfigure}[b]{\textwidth}
  \centering
  \includegraphics[width=0.81\linewidth]{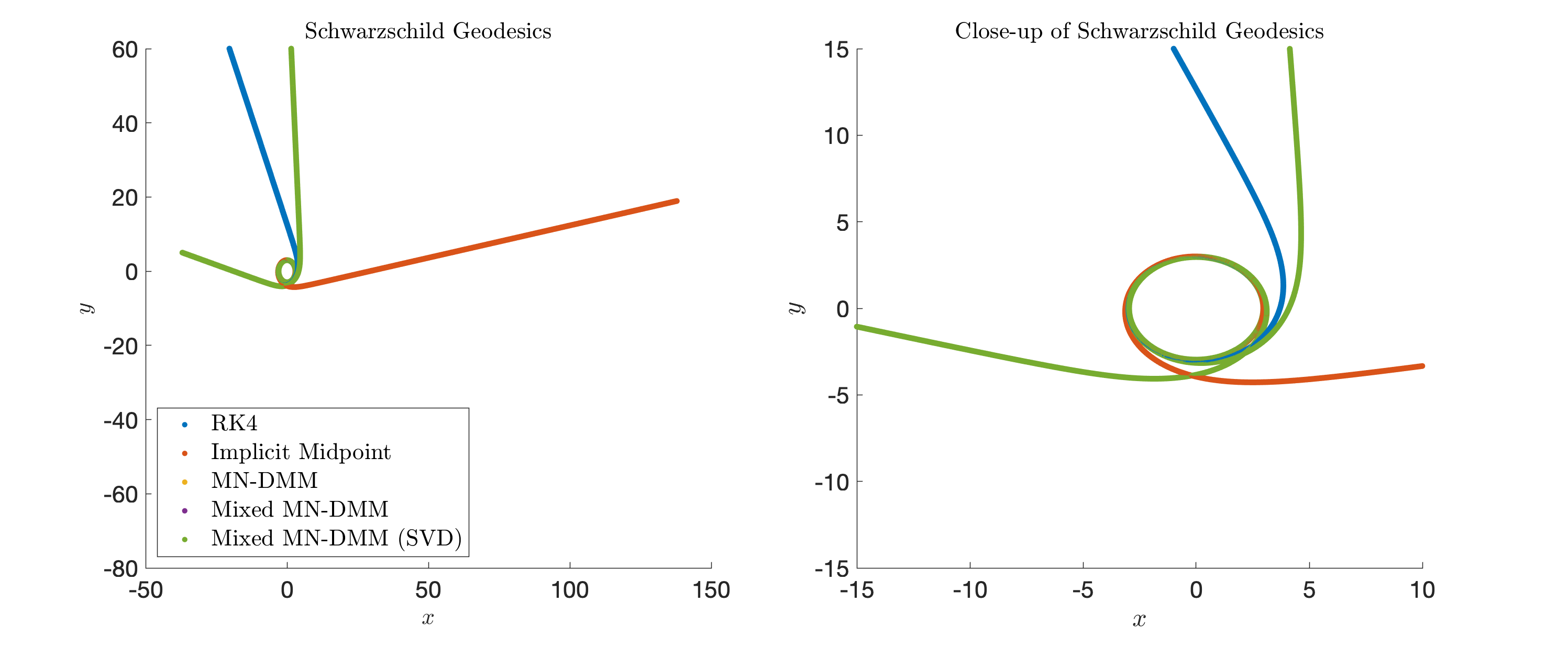}
  \vskip -2mm
  \caption{$\tau=1/3\times 2^{-5}$}
\end{subfigure}
\hfil
\begin{subfigure}[b]{\textwidth}
  \centering
  \includegraphics[width=0.81\linewidth]{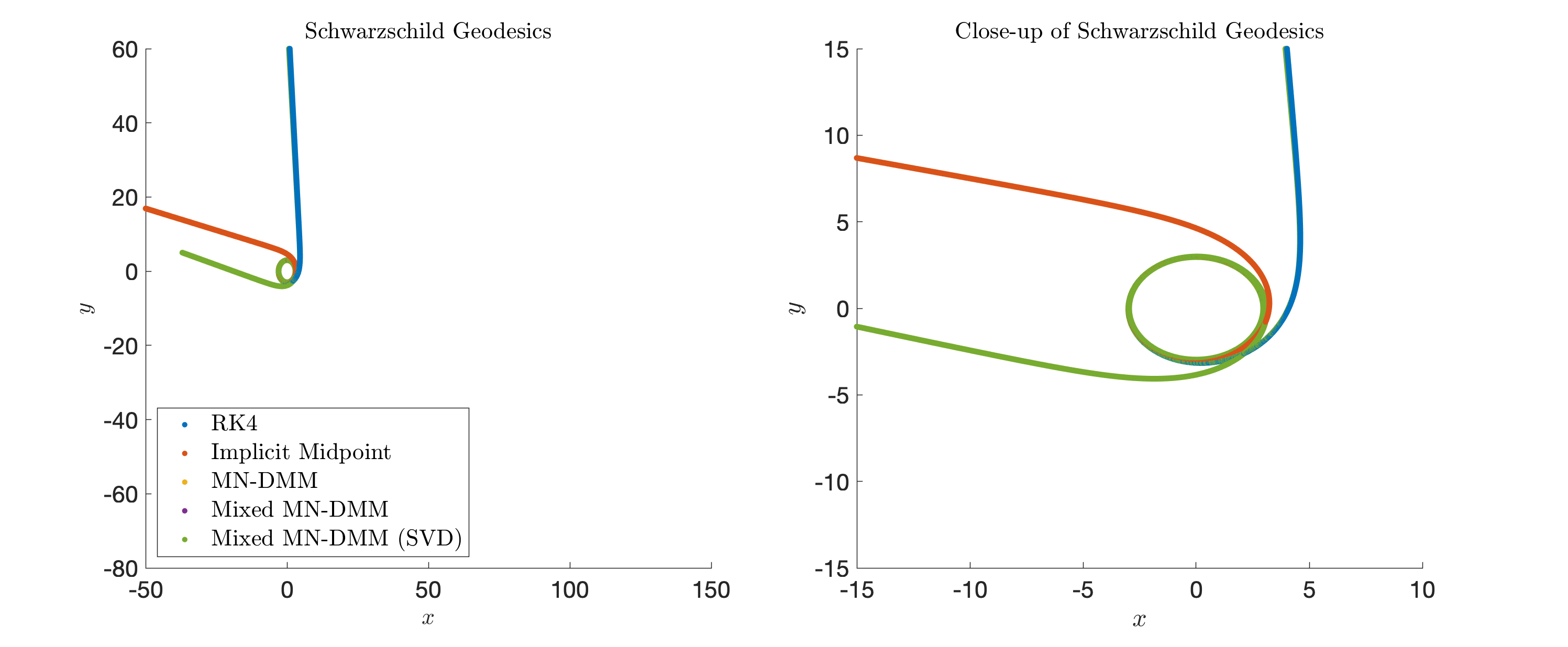}
  \vskip -2mm
  \caption{$\tau=1/3\times 2^{-7}$}
\end{subfigure}
        \caption{Comparison of geodesics curves for various $\tau$.}
        \label{fig:schwarzschild-2}
\end{figure}
\vskip -3mm
From \Cref{fig:schwarzschild-0}, we see that the Implicit Midpoint method and the different MN-DMM methods show out-going trajectories even at a relatively large time step of $\tau=1/3$. The unstable RK4 results indicate that preserving the conserved quantities near the Schwarzschild radius is critical at predicting the correct long-term trajectories. Moreover, \Cref{fig:schwarzschild-0} illustrates the geodesic curves of the Implicit Midpoint method predicts entirely wrong long term trajectory, while RK4 predicts a nonphysical outcome of ending inside the black hole.

To further study the differences between these methods at predicting the correct geodesic curves, we decrease their time step size and compare their long term trajectories and error in conserved quantities.

In \Cref{fig:schwarzschildCQ-1} and \Cref{fig:schwarzschild-2} with $\tau=1/3\times 2^{-3}$, we see that both the Implicit Midpoint method and RK4 method does not preserve conserved quantities up to machine precision, with RK4 still predicting nonphysical results. At $\tau=1/3\times 2^{-5}$, both the Implicit Midpoint method and RK4 method now predict outgoing trajectories, albeit incorrect long term trajectories. Finally at $\tau=1/3\times 2^{-7}$, the Implicit Midpoint method still predicts incorrect long term trajectory. Meanwhile, the RK4 method is now able to mimic machine precision accuracy for the conserved quantities due to its higher order accuracy and much small $\tau$. Thus, the long term trajectories of the RK4 method now agrees with the MN-DMM ones obtained using much larger $\tau$. This final example highlights that conservative integration techniques, such as the MN-DMM, can be useful in intricate short-term dynamics, where machine precision level accuracy in conserved quantities can lead to more accurate long term predictions.



\section{Acknowledgement}

The authors thank Jean-Christophe Nave and Gantumur Tsogtgerel for their valuable discussions in the initial stage of this work at McGill University. Andy T.S. Wan acknowledges funding support from the Discovery Grant program of the Natural Sciences and Engineering Research Council of Canada, (RGPIN-2019-07286) and Early Career Researcher Supplement (DGECR-2019-00467).

\bibliographystyle{unsrt}
\bibliography{Bibliography.bib}

\begin{thebibliography}{10}

\bibitem{wan2017conservative}
Andy T.~S. Wan, Alexander Bihlo, and Jean-Christophe Nave.
\newblock Conservative methods for dynamical systems.
\newblock {\em SIAM J. Numer. Anal.}, 55(5):2255--2285, 2017.

\bibitem{hair06Ay}
E.~Hairer, C.~Lubich, and G.~Wanner.
\newblock {\em {Geometric numerical integration: structure-preserving
  algorithms for ordinary differential equations}}.
\newblock Springer, Berlin, 2006.

\bibitem{MarWes01}
J.~E. Marsden and M.~West.
\newblock {Discrete Mechanics and Variational Integrators}.
\newblock {\em Acta Numerica}, 10:357--514, 2001.

\bibitem{IMNZ00}
Arieh Iserles, Hans~Z. Munthe-Kaas, Syvert~P. Nørsett, and Antonella Zanna.
\newblock Lie-group methods.
\newblock {\em Acta Numerica}, 9:215–365, 2000.

\bibitem{CalHai95}
M.P. Calvo and E.~Hairer.
\newblock {Accurate long-term integration of dynamical systems}.
\newblock {\em Appl. Numer. Math.}, 18:95--105, 1995.

\bibitem{wan2018stability}
Andy T.~S. Wan and Jean-Christophe Nave.
\newblock On the arbitrarily long-term stability of conservative methods.
\newblock {\em SIAM J. Numer. Anal.}, 56(5):2751--2775, 2018.

\bibitem{KangZaijiu95}
Feng Kang and Shang Zai-jiu.
\newblock Volume-preserving algorithms for source-free dynamical systems.
\newblock {\em Numerische Mathematik}, 71(4):451--463, 1995.

\bibitem{McL99}
G.~R. W.~Quispel R.~I.~McLachlan and N.~Robidoux.
\newblock {Geometric integration using discrete gradients}.
\newblock {\em Phil. Trans. R. Soc. Lond.}, 357:1021--1045, 1999.

\bibitem{manybody2022}
Andy T.~S. Wan, Alexander Bihlo, and Jean-Christophe Nave.
\newblock {Conservative integrators for many--body problems}.
\newblock {\em Journal of Computational Physics}, 466:111417, 2022.

\bibitem{vortex2022}
Cem Gormezano, Jean-Christophe Nave, and Andy T.~S. Wan.
\newblock {Conservative Integrators for Vortex Blob Methods on the Plane}.
\newblock {\em Journal of Computational Physics}, 469:111357, 2022.

\bibitem{HWW21}
Anil Hirani, Andy T.~S. Wan, and Nikolas Wojtalewicz.
\newblock {Conservtive Integrators for Piecewise Smooth Dynamics with
  Transversal Dynamics}.
\newblock arxiv:2106.07484, 2021.

\bibitem{chmc22}
Geoffrey McGregor and Andy T.~S. Wan.
\newblock {Conservtive Hamiltonian Monte Carlo}.
\newblock arxiv:2206.06901, 2022.

\bibitem{Bjorck96}
{\AA}ke Bj\"{o}rck.
\newblock {\em Numerical methods for least squares problems}.
\newblock Society for Industrial and Applied Mathematics (SIAM), Philadelphia,
  PA, 1996.

\bibitem{quarteroni2010numerical}
Alfio Quarteroni, Riccardo Sacco, and Fausto Saleri.
\newblock {\em Numerical Mathematics}, volume~37 of {\em Texts in Applied
  Mathematics}.
\newblock Springer-Verlag, Berlin, second edition, 2007.

\bibitem{trefethenBau97}
Lloyd~N. Trefethen and David Bau.
\newblock {\em Numerical Linear Algebra}.
\newblock SIAM, 1997.

\bibitem{schi03a}
R.~Schimming.
\newblock {Conservation laws for Lotka--Volterra models}.
\newblock {\em Math. Methods Appl. Sci.}, 26(17):1517--1528, 2003.

\bibitem{AblowitzSegur81}
Mark~J. Ablowitz and Harvey Segur.
\newblock {\em Solitons and the inverse scattering transform}.
\newblock SIAM: Studies in applied and numerical mathematics, 1981.

\bibitem{kus83}
M.~Kus.
\newblock {Integrals of motion for the Lorenz system}.
\newblock {\em J. Phys. A: Math. Gen.}, 16(18):L689--L691, 1983.

\bibitem{carroll_2019}
Sean~M. Carroll.
\newblock {\em Spacetime and Geometry: An Introduction to General Relativity}.
\newblock Cambridge University Press, 2019.

\bibitem{Godinho_2014}
Leonor Godinho and Jos\'{e} Nat\'{a}rio.
\newblock {\em An introduction to {R}iemannian geometry}.
\newblock Universitext. Springer, Cham, 2014.
\newblock With applications to mechanics and relativity.

\bibitem{o1983semi}
Barrett O'Neill.
\newblock {\em Semi-Riemannian geometry with applications to relativity}.
\newblock Academic press, 1983.

\end{thebibliography}

\section{Appendix}

\end{document}